\documentclass[11pt,reqno]{amsart}
\usepackage[a4paper]{anysize}\marginsize{2.5cm}{2.5cm}{2cm}{2cm}
\pdfpagewidth=\paperwidth \pdfpageheight=\paperheight
\allowdisplaybreaks
\usepackage{amsmath, amsthm, amsfonts, amssymb, color}
\usepackage[colorlinks=true,linkcolor=blue,urlcolor=blue]{hyperref}
\usepackage{url}
\usepackage{graphicx}
\usepackage{caption}
\usepackage{mathtools}
\usepackage{enumerate}
\usepackage[dvipsnames]{xcolor}
\usepackage{verbatim}
\usepackage{tikz,tikz-cd,tikz-3dplot}
\usepackage{amssymb}
\usetikzlibrary{matrix}
\usetikzlibrary{arrows}
\usepackage{algorithm}
\usepackage[noend]{algpseudocode}
\usepackage{caption}
\usepackage[normalem]{ulem}
\usepackage{subcaption}
\usepackage{makecell}
\usepackage{array}
\usepackage[scr=rsfs]{mathalpha}
\usepackage{multirow,array}
\usepackage{listings}

\usepackage{cleveref}
 
\usepackage{mathdots}
\usepackage{pdfpages}
\usepackage{booktabs}
\usepackage{mathrsfs,hyperref} 
\usepackage[numbers,sort&compress]{natbib}

\DeclareFontEncoding{LS1}{}{}
\DeclareFontSubstitution{LS1}{stix}{m}{n}
\DeclareSymbolFont{symbols2}{LS1}{stixfrak} {m} {n}
\DeclareMathSymbol{\operp}{\mathbin}{symbols2}{"A8}

\newtheorem{theorem}{Theorem}

\newtheorem{proposition}[theorem]{Proposition}
\newtheorem{lemma}[theorem]{Lemma}
\newtheorem{corollary}[theorem]{Corollary}

\theoremstyle{definition}
\newtheorem{definition}[theorem]{Definition}

\newtheorem{remark}[theorem]{Remark}
\newtheorem{conjecture}[theorem]{Conjecture}

\newtheorem{example}[theorem]{Example}

\numberwithin{theorem}{section}

\newcommand{\CC}{\mathbb{C} }

\newcommand{\sign}{\mathrm{sign}}
\newcommand{\codim}{\mathrm{codim}}
\newcommand{\Jac}{\mathrm{Jac}}
\newcommand{\join}{\text{Join}}
\newcommand{\CCdeg}{\operatorname{CCdeg}}
\DeclareMathOperator{\Pf}{\bf Pf}
\DeclareMathOperator{\GL}{GL}
\DeclareMathOperator{\SL}{SL}

\numberwithin{equation}{section}

\lstdefinelanguage{Julia}%
  {morekeywords={abstract,break,case,catch,const,continue,do,else,elseif,%
      end,export,false,for,function,immutable,import,importall,if,in,%
      macro,module,otherwise,quote,return,switch,true,try,type,typealias,%
      using,while},%
   sensitive=true,%
   morecomment=[l]\#,%
   morecomment=[n]{\#=}{=\#},%
   morestring=[s]{"}{"},%
   morestring=[m]{'}{'},%
}[keywords,comments,strings]%

\title[On the Coupled Cluster Doubles Truncation Variety of Four Electrons]{On the Coupled Cluster Doubles Truncation Variety of Four Electrons}

\author[F. M. Faulstich]{Fabian M. Faulstich}
\address{Rensselaer Polytechnic Institute, 110 Eighth Street, Troy, NY 12180, USA}
\email{faulsf@rpi.edu}

\author[V. Galgano]{Vincenzo Galgano}
\address{Max Planck Institute of Molecular Cell Biology and Genetics, Dresden, Germany; Faculty of Mathematics, Technische Universität Dresden, Germany}
\email{galgano@mpi-cbg.de}

\author[E. Neuhaus]{Elke Neuhaus}
\address{Max Planck Institute for Mathematics in the Sciences, Leipzig, Germany}
\email{elke.neuhaus@mis.mpg.de}

\author[I. Portakal]{Irem Portakal}
\address{Max Planck Institute for Mathematics in the Sciences, Leipzig, Germany}
\email{mail@irem-portakal.de}

\begin{document}

\begin{abstract}
We extend recent algebro-geometric results for coupled cluster theory of quantum many-body systems to the truncation varieties arising from the doubles approximation (CCD), focusing on the first genuinely nonlinear doubles regime of four electrons. Since this doubles truncation variety does not coincide with previously studied varieties, we initiate a systematic investigation of its basic algebro-geometric invariants. We show that for $4$ electrons in $n \geq 13$ orbitals, the CCD truncation variety is not a complete intersection, and we present numerical and symbolic evidence that for the remaining cases it is a complete intersection defined by $\binom{n-4}{4}$ quadrics. Using representation-theoretic arguments, we uncover a Pfaffian structure governing the quadratic relations that define the truncation variety for any $n$, and show that an exact tensor product factorization holds in the distinguished limit of disconnected doubles. We connect these structural results to the computation of the beryllium insertion into molecular hydrogen ({Be$\cdots$H$_2$ $\to$ H--Be--H}), a small but computationally challenging bond formation process where multiconfigurational effects become pronounced. \\
\hfill\break
\hfill\break
\end{abstract}

\maketitle

\section{Introduction}
\noindent
Computational (quantum) chemistry provides a quantitative link between quantum mechanics and chemical observables by translating microscopic models into algorithms for energies, forces, and derived properties.  A central component is {\it electronic structure simulation}, which turns the many-electron problem into a sequence of large-scale numerical tasks and thereby exposes deep questions in approximation theory, numerical linear and nonlinear algebra~\cite{helgaker2014molecular}.

The fundamental mathematical model is the {\it quantum many-body problem} for electrons moving in an external potential generated by the nuclei. The central quantity is a many-electron wave function, i.e., a function of $3d$ spatial variables (and spin) for $d$ electrons, constrained by fermionic antisymmetry. Even before numerical approximation enters, this combination of high dimensionality and antisymmetry makes the electronic structure problem a prototypically difficult computational task. The high fidelity requirements of electronic structure simulations demand an accurate representation of the nontrivial coupling created by electron--electron interactions in a very high dimensional antisymmetric space, whereas numerical tractability requires advanced techniques that exploit problem-specific structures.

In practice, one introduces a finite representation, which turns the electronic Schr\"odinger equation into a large eigenvalue problem on an antisymmetric tensor space. The dimension of this space grows combinatorially with the number of orbitals; hence, the direct application of ``standard'' numerical approaches quickly becomes infeasible. This scaling barrier motivates working with reduced model classes for the wave function, but these classes are typically nonlinear. The reason is that electron--electron interactions mix a large number of basis configurations, so that achieving a given accuracy with a linear approximation space generally forces the dimension of that space to grow rapidly. Nonlinear parametrizations offer an alternative reduction mechanism by representing the relevant states as a low-dimensional manifold or variety inside the full antisymmetric space, thereby reducing the effective dimension of the problem. From this viewpoint, much of electronic structure theory can be seen as the design and analysis of computationally accessible model classes that capture the essential physics.

One instance of such a nonlinear reduction strategy is provided by {\it coupled cluster theory}, which is highly successful and widely regarded in the computational chemistry community as a benchmark for high-accuracy electronic structure calculations. In this approach, one does not solve the discretized eigenvalue problem directly in the full antisymmetric space. Instead, one introduces a parametrized family of trial states obtained from a fixed reference vector by applying an exponential map to a structured operator, and one determines the parameters by enforcing the eigenvalue equation through a set of projected residual conditions. Practical variants arise by restricting the parametrizing operator to a prescribed subset of degrees of freedom, which yields a smaller but nonlinear system. These restrictions make large scale computations feasible, but they also raise mathematical questions about the solution set of the nonlinear system, such as the presence of multiple solutions, singular points, and the geometry of the feasible set.

An exterior algebra and algebraic geometry viewpoint on coupled cluster theory was recently initiated in ~\cite{faulstich2024coupled} and developed further in ~\cite{faulstich2024algebraic}. A key outcome is the notion of \emph{coupled cluster truncation varieties}, which model the feasible set determined by truncating the exponential parametrization \eqref{eq: exponential parametrization} and the associated unlinked coupled cluster equations~\eqref{eq:truncatedproblem}. A first focused algebraic investigation of truncation varieties was carried out in \cite{borovik2024coupled}, building on the observation that the coupled cluster singles truncation variety (only one-electron excitations are included in the cluster operator) recovers the Grassmannian, a classical and well-understood variety. 

The present manuscript continues this line of research with an emphasis on the coupled cluster doubles (CCD) truncation variety, where only two-electron excitations are included in the cluster operator. Unlike the singles case, the doubles truncation varieties do not appear to coincide with specifically studied varieties, and their geometry must therefore be analyzed directly.

In Section~\ref{Section 2}, we describe the defining quadrics $P_J$ in \eqref{eq:PJ} and the linear equations of the affine CCD truncation variety $V_{\{2\}}^{\mathbb A}$ for four electrons, expressed in coordinates of $6\binom{n-4}{2}$ quantum states~\eqref{eq: quantum states}. Constructing the defining ideal of the projective CCD variety $V_{\{2\}}$ requires saturating an ideal generated by $\binom{n-4}{4}$ homogenized quadrics $Q_J$ \eqref{homogenized quadrics}, a process that becomes computationally prohibitive for $n>10$ spin orbitals (see Example~\ref{ex: 8 and 9 spin orbitals}). This motivates our algebro-geometric study of the varieties $V_P$ and $V_Q$ defined by the quadrics $P_J$ and $Q_J$, respectively. The variety $V_P$ is a complete intersection w.r.t the quadrics $P_J$ if and only if the projective truncation variety $V_{\{2\}}$ is isomorphic to $V_Q$ (Proposition~\ref{ci iff sat}). In this case, the degree of $V_{\{2\}}$ is $2^{\binom{n-4}{4}}$ (Corollary~\ref{cor: degree truncation}). We show that for $n\geq 13$, $V_{\{2\}}$ is not a complete intersection w.r.t the defining quadrics (Proposition~\ref{prop: n13}). Utilizing homotopy iterators \cite{breiding2025homotopyiterators}, we present numerical evidence that $V_{\{2\}}$ is a complete intersection for $n\leq 12$ in Section~\ref{sec: numerical evidence}. 

Section~\ref{sec: quadrics PJs and pfaffians} analyzes the geometry of $V_P$ using representation theory (Section~\ref{subsec:preliminaries modules}). We decompose the space containing the quadrics $P_J$ into irreducible representations of $\mathrm{SL}_{4}\times\mathrm{SL}_{n-4}$ (Section~\ref{subsec: decomposition})  and identify which irreducible components contain the linear span $\langle P_J\mid J\rangle_{\mathbb{C}}$ as a $\mathfrak{S}_4\times\mathfrak{S}_{n-4}$-subrepresentation. This allows us to conclude that the quadrics $P_J$ are linear combinations of tensor products of Pfaffians in Theorem~\ref{thm:Pfaffians}. We also address the natural question whether the quadrics $P_J$ can be pure tensor products of Pfaffians (and not linear combination of them), showing that this happens in the framework of disconnected doubles (see Section~\ref{subsec:disconnected}). 

In Section~\ref{sec:ccd}, we conjecture the coupled cluster degree supported by numerical computations (Conjecture~\ref{conj: CC degree}) and compare it to an upper bound computed using \cite[Theorem~5.2]{faulstich2024algebraic}. In Section~\ref{Section 5}, we conclude the paper with a quantum chemistry simulation for four electrons in up to ten spin-orbitals, which is a CAS(2,5) setting. We investigate the beryllium insertion into molecular hydrogen, a challenging system since it involves solution crossings along the pathway~\cite{purvis1983c2v}. This example pushes the algebraic pipeline to a regime where the coupled-cluster polynomial system admits millions of solutions for the corresponding generic Hamiltonian. It also provides a concrete test of how reference crossings reshape the distribution of coupled-cluster roots and the fraction of roots that yield real-valued energies. Moreover, our results motivate studying not only the root structure itself, but also how a distinguished observable, here the energy, organizes the solution set as the system changes. All numerical and symbolic computations in this paper were carried out using \texttt{Julia} and \texttt{Macaulay2}. The corresponding code is accessible in \cite{Faulstich2025SupplementaryCode}.

\subsection{Previous works and perspective}

The first systematic study of the root structure of the CC equations dates back to 1978, when \v{Z}ivkovi\'c and Monkhorst analyzed singularities and multiple solutions in single-reference CC theory~\cite{vzivkovic1978analytic}. In the early 1990s, Paldus and coauthors developed further mathematical and numerical analyses of the solution manifolds of single-reference and state-universal multireference CC equations, including their singularities and analytic properties~\cite{paldus1993application,piecuch1990coupled}. Homotopy methods were revived in the CC context by Kowalski and Jankowski, who solved a CCD (CC doubles) system in~\cite{kowalski1998towards}; Kowalski and Piecuch subsequently extended these ideas to CCSD (CC singles, and doubles), CCSDT (CC singles, doubles, and triples), and CCSDTQ (CC singles, doubles, triples, and quadruples) for a four-electron minimal-basis model~\cite{piecuch2000search}. These works led to the $\beta$-nested equations and the \emph{Fundamental Theorem of the $\beta$-NE Formalism}~\cite{kowalski2000complete2}, which explains how solution branches connect across truncation levels (e.g., CCSD $\to$ CCSDT $\to$ CCSDTQ). The resulting \emph{Kowalski--Piecuch homotopy} has recently been analyzed in depth using topological degree theory~\cite{csirik2023disc,csirik2023coupled}. Related homotopy-based studies produced complete solution sets of the generalized Bloch equation~\cite{kowalski2000complete} and clarified symmetry-breaking mechanisms within CC theory~\cite{kowalski1998full,jankowski1999physical1,jankowski1999physical2,jankowski1999physical3,jankowski1999physical4}. For a recent perspective on homotopy methods in quantum chemistry, see \cite{faulstich2023homotopy}.

A more explicitly algebraic and mathematically driven line of work has recently gained momentum. Initiated by some of the present authors, an algebraic perspective on CC theory was introduced in \cite{faulstich2024coupled}, and subsequent in-depth analysis led to the algebraic framework of CC truncation varieties~\cite{faulstich2024algebraic}. This viewpoint has since motivated a sequence of contributions ranging from pure mathematics~\cite{borovik2024coupled,sverrisdottir2025algebraic,price2026plane} to domain applications targeting the description of excited states~\cite{faulstich2026algebraic} or spin-adapted coupled cluster theory \cite{sverrisdottir2024exploring}.
More broadly, these developments reflect a growing interest in a rigorous mathematical understanding of CC theory within the pure and applied mathematics community. In this context, Schneider gave the first local analysis of CC theory in 2009 using functional-analytic techniques~\cite{schneider2009analysis}; this framework was generalized in~\cite{rohwedder2013continuous,rohwedder2013error} and extended to additional CC variants in~\cite{laestadius2018analysis,laestadius2019coupled,faulstich2019analysis}. A complementary and more flexible approach based on topological degree theory was later proposed by Csirik and Laestadius~\cite{csirik2023disc,csirik2023coupled}. The most recent numerical-analysis results for single-reference CC were obtained by Hassan, Maday, and Wang~\cite{hassan2023analysis,hassan2023analysis2}. For a comprehensive review of the recent mathematical advances, we refer the interested reader to~\cite{faulstich2024recent}. Beyond the direct applications to coupled cluster theory, algebraic and computational approaches find more and more use in electronic structure theory~\cite{gontier2019numerical,cances2024mathematical,pokhilko2025homotopy,harwood2022improving,borovik2025numerical}.

\subsection{The electronic structure problem and coupled-cluster theory}
\label{sec:background}

We briefly review the electronic Schrödinger equation and the coupled cluster formalism to fix notation and ensure a self-contained presentation. For a more detailed exposition, we refer the interested reader to~\cite{helgaker2014molecular}. 
The {\em electronic Schrödinger equation} describing $d$ electrons in $n \geq 2d$ spin-orbitals in an electro-magnetic field is an eigenvalue problem~\cite{schrodinger1926undulatory}
\begin{equation} 
\label{eq:schroedinger}
H \Psi
=
\lambda \Psi,
\end{equation}
where the unknown is the {\em wave function} $\Psi(x_1,x_2,\ldots,x_d)$. 
The \emph{Hamiltonian} $H$ does not admit a closed-form solution beyond the hydrogen atom and requires a finite representation that is amenable to numerical simulations. The solutions to the many-electron Schrödinger equation have to be anti-symmetric functions; this is axiomatically rooted in quantum theory, going back to Pauli's exclusion principle~\cite{pauli1925zusammenhang}. Given a set of sufficiently smooth univariate functions $\{\chi_i\}_{i=1}^n$ defining a vector space $V$, we construct the Galerkin space 
\begin{equation*}
\bigwedge^dV\simeq \bigwedge^d \mathbb{R}^n = \mathcal{H} 
\end{equation*}
with its standard basis vectors 
\begin{equation*}
\chi_I
:=
\chi_{i_1} \wedge \chi_{i_2} \wedge  \cdots \wedge \chi_{i_d}
\leftrightarrow
e_I := e_{i_1} \wedge e_{i_2} \wedge  \cdots \wedge e_{i_d},
\end{equation*} 
for (increasingly ordered) subsets $I = \{i_1,i_2,\ldots,i_d\} \in \binom{[n]}{d}$. Taking specific considerations regarding the univariate functions into account, see~\cite[Section 4]{faulstich2024algebraic}, the first basis vector $e_{[d]}$ for $[d] := \{1,2,\ldots,d\}$ takes a special role, and it is referred to as the {\em reference state}. General vectors $\psi \in \mathcal{H}$ are called {\em (quantum) states} and they are written uniquely as linear combinations of the basis vectors, i.e., 
\begin{equation}\label{eq: quantum states}
\psi 
= 
\sum_{I \in \binom{[n]}{d}} \psi_I  e_I.
\end{equation}

\noindent We note that, up to a sign, every basis vector $e_I$ can be obtained by replacing $k\in [d]$ of the vectors $e_{\ell}\in\mathbb{R}^n$ in $e_{[d]}$. We define this operation for a single replacement as a concatenation of the interior product with the exterior product. More precisely, let $i \in [d]$ and $a \in [n] \setminus [d]$, then the replacement of $i$ by $a$ is represented as
\begin{equation}\label{eq: def of X_ij}
X_{i,a} e_I :=  e_a \wedge (e_i \lrcorner e_I).
\end{equation}

\begin{example}[$d=2,n\geq4 $] 
\label{ex:BasisReplacement}
We obtain that 
$e_{[2]} = e_{\{1,2\}} = e_1 \wedge e_2$
and 
\begin{equation*}
e_{\{2,3\}} 
= e_2 \wedge e_3
= -e_3 \wedge e_2
= -e_3 \wedge (e_1 \lrcorner (e_1 \wedge e_2))
=-e_3 \wedge (e_1 \lrcorner e_{\{1,2\}})
= - X_{1,3} e_{\{1,2\}},
\end{equation*}
where $X_{1,3}$ replaces index 1 with 3.
\end{example}

\noindent We refer to the replacement as {\it $i$-$a$-excitation} and the corresponding matrix $X_{i,a}$ is called the $i$-$a$-excitation matrix (also called an elementary particle-hole excitation operator~\cite{vcivzek1966correlation}). 

\begin{example}[$d=2,n=5$] 
\label{ex:2in5}
For 2 electrons in 5 spin-orbitals, there are $10$ two-element subsets of $[5]$:
\begin{equation*}
{{[5]} \choose {2}}
=
\left\lbrace
\{1,2\},\{1,3\},\{1,4\},\{1,5\},
\{2,3\},\{2,4\},\{2,5\},
\{3,4\},\{3,5\},
\{4,5\}
\right\rbrace.
\end{equation*}
The excitation matrix $X_{2,4}$ we obtain for exchanging $e_2$ for $e_4$ is given by
\begin{equation*}
X_{2,4}
=
\begin{bmatrix}
0 &  0 & 0 & 0 & 0 &    0 &    0 &  0 & 0 & 0 \,\,\, \\
0 &  0 & 0 & 0 & 0 &    0 &    0 &  0 & 0 & 0 \,\,\, \\
1 &  0 & 0 & 0 & 0 &    0 &    0 &  0 & 0 & 0 \,\,\, \\
0 &  0 & 0 & 0 & 0 &    0 &    0 &  0 & 0 & 0 \,\,\, \\
0 &  0 & 0 & 0 & 0 &    0 &    0 &  0 & 0 & 0 \,\,\, \\
0 &  0 & 0 & 0 & 0 &    0 &    0 &  0 & 0 & 0 \,\,\, \\
0 &  0 & 0 & 0 & 0 &    0 &    0 &  0 & 0 & 0 \,\,\, \\
0 &  0 & 0 & 0 &-1 & 0 &   0 &  0 & 0 & 0 \,\,\, \\
0 &  0 & 0 & 0 & 0 & 0 &   0 & 0 & 0 & 0 \, \,\,\\
0 &  0 & 0 & 0 & 0 & 0 & 1 & 0 & 0 & 0 
\end{bmatrix}.
\end{equation*}  
Its rows and columns are labeled by the ordered sets in ${{[5]} \choose {2}}$.
\end{example}

\noindent This concept can be generalized to multiple index replacements, i.e., let $\alpha \subseteq [d]$ and $\beta \subseteq [n] \setminus [d]$ with $|\alpha| = |\beta| = k$, then 
\begin{equation}\label{eq: def of X_IJ}
X_{\alpha, \beta} = \prod_{j=1}^k X_{\alpha_j,\beta_j}.
\end{equation}
This uniquely identifies every basis element $e_I$ with an $\alpha$-$\beta$-excitation upto a sign. Hence, 

\begin{equation}
\label{eq:FCI}
\psi 
= 
\sum_{I \in \binom{[n]}{d}} \psi_I  e_I 
=
\sum_{\alpha, \beta}
c_{\alpha,\beta} X_{\alpha, \beta}e_{[d]}
=
\Omega e_{[d]}
\end{equation}

where $\alpha \subseteq [d]$ and $\beta \subseteq [n] \setminus [d]$ with $|\alpha| = |\beta| = k$. The coefficients $c_{\alpha,\beta}$ are known as {\em configuration interaction coefficients} in quantum chemistry that define the wave operator $\Omega$. The cardinality of $\alpha$ plays an important role; it is known as the {\it excitation rank} or \emph{level}. The identification between these two systems of coordinates on $\mathcal{H}$ is  as follows:
 \begin{equation}
 \label{eq:cpsi}  
 c_{\alpha,\beta} = \psi_I
 \quad {\rm where} \quad 
 I = ([d] \backslash \alpha) \cup \beta. 
 \end{equation}

\noindent A key observation is that excitation matrices $X_{\alpha, \beta}$ are nilpotent to order two which yields that 
\begin{equation*}
T 
= \log(\Omega)
= \sum_{k=0}^\infty \frac{(-1)^k}{k+1} (\Omega-Id)^{k+1} 
= \sum_{k=0}^d \frac{(-1)^k}{k+1} (\Omega-Id)^{k+1},
\end{equation*}
known as the cluster operator, is well-defined~\cite{faulstich2024coupled}. We moreover note that the linear hull of cluster matrices is closed under multiplication by construction (c.f. \eqref{eq: def of X_IJ}). Hence, 
\begin{equation*}
T(t)
= \sum_{\alpha, \beta} t_{\alpha,\beta} X_{\alpha, \beta}
\end{equation*}
giving rise to the \emph{cluster amplitudes} $t_{\alpha,\beta}$. Note that, just as we have two systems of coordinates $\{\psi_I\}$ and $\{c_{\alpha, \beta}\}$ parameterizing the quantum states in \eqref{eq:FCI}, there is an equivalent representation of the cluster amplitudes as 
\begin{equation*}
 t_{\alpha,\beta} = x_I
 \quad {\rm where} \quad 
 I = ([d] \backslash \alpha) \cup \beta. 
\end{equation*}
Just like the quantum states, the cluster amplitudes live in a vector space isomorphic to $\bigwedge^d \mathbb R^n$.
Put differently, the cluster operator $T(x)$ is a matrix of size ${n \choose d} \times {n \choose d}$ indexed by ${[n] \choose d}$ which has as entries polynomials in $x$ and is nilpotent of order $d+1$. We conclude that an intermediately normalized quantum state $\psi$, i.e., $\langle \psi ,e_{[d]} \rangle = 1$, can be equivalently represented via 
\begin{equation}\label{eq: exponential parametrization}
  \psi = \exp(T(x))e_{[d]},
\end{equation}
which is known as the \emph{exponential parametrization} of $\psi$.

For more background on the construction of the matrix $T$, we refer to \cite{faulstich2024algebraic}. Note that in the image of the exponential map,  $\psi_{[d]}=1$. Furthermore, the exponential map is birational; therefore we can also express the cluster amplitudes as polynomials in the quantum states $x = x(\psi)$. Substituting the ansatz $\psi=e^{T} e_{[d]}$ in the Schrödinger equation \eqref{eq:schroedinger} yields 
\begin{equation*}
H e^{T} e_{[d]} = \lambda e^{T} e_{[d]}
\quad \Leftrightarrow \quad 
\left\lbrace 
\begin{aligned}
e_{[d]}^\top e^{-T} H e^{T} e_{[d]} &= \lambda, \\
e_{I}^\top e^{-T} H e^{T} e_{[d]} &= 0~\forall I.
\end{aligned}
\right.
\end{equation*}
The latter is known as \emph{the coupled cluster equations}. 

Since these systems are very large and computationally very expensive, one can impose special constraints on the coupled cluster equations. Instead of considering the linear system in ${n \choose d}$ variables and ${n \choose d}$ equations, one can truncate the problem to a smaller, nonlinear one:
\begin{equation}\label{eq:truncatedproblem}
   (H \psi)_{\sigma} = \lambda \psi_{\sigma}, \quad \psi \in V_{\sigma}. 
\end{equation}
Here, $\sigma \subsetneq [d]$ is a non-empty set of possible excitation ranks and $(\cdot)_{\sigma}$ denotes the projection to coordinates indexed by level in $\sigma$. Moreover, the feasible set is restricted to the so-called \emph{truncation variety} denoted by $V_{\sigma}$. It is defined as the projectivization of the variety in $\mathcal H$ obtained by truncating the exponential parametrization to a certain subspace.

\begin{remark}
In order to use techniques from algebraic geometry, it is often preferable to work over the complex field. In this respect, from this moment on, we complexify all real vector spaces introduced so far. Projective spaces and algebraic varieties will be over $\mathbb{C}$.
\end{remark}

Let $\mathbb C[ \bigwedge^d\mathbb C^n]=\mathbb C[x_I \mid I \in \binom{[n]}{d}]$ be the coordinate ring of $\bigwedge^d \mathbb C^n$. Consider the linear subspace 
\begin{align*}
\mathcal V_{\sigma} & := \textrm{Span}( e_I \mid |I \setminus [d]| \in \sigma \cup \{0\} ) = \textrm{V}(x_J \mid |J \setminus [d] | \in [d] \setminus [\sigma]) \subseteq \bigwedge^d \mathbb C^n
\end{align*}
spanned by all basis vectors of level in $\sigma$ and cut out by the vanishing of $x_J$ of level not in $\sigma$.

\begin{definition}
 The image of the linear subspace $\mathcal V_{\sigma}$ under the injective exponential parametrization $\bigwedge^d\mathbb C^n \rightarrow \mathcal{H}$ defined by \eqref{eq: exponential parametrization} is called the \emph{affine truncation variety} denoted by {\relpenalty=10000\binoppenalty=10000 $V_{\sigma}^{\mathbb A} \subset \mathcal H \simeq \mathbb C^{n \choose d}$}. The \emph{(projective) truncation variety} $V_{\sigma}$ is defined as the projective closure of $V_{\sigma}^{\mathbb A}$ in {\relpenalty=10000\binoppenalty=10000$\mathbb P(\mathcal H) = \mathbb P^{{n \choose d}-1}$}.
\end{definition}

\noindent We observe that $V_{\sigma} \cap \{ \psi_{[d]} =1\} = V_{\sigma}^{\mathbb A}$.
Since the exponential parametrization is rational, its dimension is the same as the dimension of $\mathcal V_{\sigma}$. By abuse of notation, we will write $\mathbb C[\bigwedge^d\mathbb C^n]$ also in place of $\mathbb C[\mathcal{H}]= \mathbb C[\psi_I \mid I \in \binom{[n]}{d}]$.

For $\sigma = \{1\}$, the coupled-cluster (CC) singles truncation variety $V_{\{1\}}$ coincides with the Grassmannian $\mathrm{Gr}(d, n)$~\cite{borovik2024coupled}. In practice, CC doubles ($\sigma = \{2\}$) and CC singles doubles ($\sigma = \{1,2\}$) are the prevalent choices of truncation. We will focus on the coupled-cluster doubles (CCD) truncation variety $V_{\{2\}}$.
By \cite[Theorem~3.10]{faulstich2024algebraic}, for $d \leq 3$, $V_{\{2\}}$ is linear. 
This paper aims to analyze $V_{\{2\}}$ in the first non-linear case $d = 4$, i.e.\ for four electrons.

\section{The coupled cluster doubles truncation variety}\label{Section 2}

This section studies the coupled cluster doubles truncation variety $V_{\{2\}}$ (henceforth referred to simply as the truncation variety) in the first genuinely nonlinear regime, namely, $d=4$. Doubles truncation already yields a rich projective variety whose geometry, to the best of our knowledge, does not coincide with a previously studied class. Our goal is twofold: First, we construct $V_{\{2\}}$ explicitly from the truncated exponential parametrization and derive defining equations in coordinates organized by excitation level. Second, we use these equations to analyze algebraic and geometric features that directly affect computation, including homogenization and saturation, complete intersection behavior, and structural properties of the quadratic generators. 

\subsection{Construction for four electrons}

Recall that the level of a coordinate $\psi_{I}$ is the cardinality of $I \setminus [4]$. In general, following \cite[Theorem 2.5]{faulstich2024algebraic}, the coordinates of the exponential parametrization of a certain level are all replicas of some master polynomial of this level. More precisely, the master polynomial of level $d$ is $c_{[d], [n]\setminus [d]}(t)$ where $n = 2d$, i.e., 
\begin{equation*}
c_{[d], [n]\setminus [d]}(t)= \sum\limits_{\pi} \sign(\pi) t_{\pi},
\end{equation*}
where the sum goes over all uniform block permutations 
$ \pi = \{ \alpha_1 \cup \beta_1, \ldots, \alpha_k \cup \beta_k \}$
of $[2d]$
with $\alpha_i \subset [d]$, $\beta_i \subset [2d] \setminus [d]$ and $|\alpha_i|=|\beta_i|$
and where 
$t_{\pi} = t_{\alpha_1, \beta_1} \cdots t_{\alpha_k ,\beta_k}$ 
and the sign of $\pi$ is the product of the signs of the permutations 
$[d] \mapsto (\alpha_1, \ldots, \alpha_k)$ and $[2d] \setminus [d] \mapsto (\beta_1, \ldots, \beta_k)$. Similarly, we obtain representations $\psi_J(x)$ that are replicas of master polynomials in $x$. These replicas $\psi_J$ 
then arise from the master polynomials by changing indices and possibly the sign. This sign is the same as the sign of 
$x_J$ in the first column of the matrix $T$ (see Section~\ref{sec:background}). This is due to the fact that 
$x_J$ appears linearly in 
$\psi_J(x)$, originating from the degree one term in the sum formula of the exponential parametrization and from the fact that the linear term always has a positive sign in the master polynomial. 

The entries of $T$ in column $I=[4]$ and a row $J$ are given by $\pm x_J$ or $\pm t_{I \setminus J, J \setminus I}$ with the sign determined by
the sign of the permutation $\pi_J: I \mapsto (I \cap J, I \setminus J)$ times the sign of the permutation $\pi_I: J \mapsto (I \cap J, J \setminus I)$. 
For $I=[4]$ and $J$ of level $4$, both of these permutations are the identity.
For $I=[4]$ and $J$ of level 2, $\sign(\pi_I) = \sign (\text{id}) =1$ and 
\begin{equation}
\sign(\pi_J)= \left\{ \begin{array}{cl}
    1, & J \cap [4] = \{1,2\}, \{1,4\}, \{2,3\}, \{3,4\}, \\
    -1,  & J \cap [4] = \{1,3\}, \{2,4\}.
\end{array} \right.
\end{equation}\label{eq: sign of pi}
Note that for $d=4$ electrons, the master polynomials for levels $1$ and $3$ contain, by construction, a level $1$ or level $3$ variable in every term. Hence, when truncating at level $2$, these terms vanish. At level $2$ we obtain ${4 \choose 2}{n-4 \choose 2}$ replicas of the master polynomial 
$\psi_{34}(x) =  x_{14}x_{23} - x_{13}x_{24} + x_{34}$
with the sign adjusted as in \eqref{eq: sign of pi}. Thus, truncating to level $2$ yields the level $2$ linear terms $\pm x_J$.
At level~$4$, after truncating to level $2$, we are left with $n - 4 \choose 4$ quadrics $P_J(x)$ in $\mathbb C[\bigwedge^4\mathbb C^n]$
where $J:=\{j_1 < j_2 < j_3 < j_4\} \in \binom{[n] \setminus [4]}{4}$.
Since for $J$ of level 2 it is $\psi_I = \pm x_I$ and since the monomials of the quadrics are $x_{I_1}x_{I_2}$ for $I_1,I_2$ of level 2 with $\sign(\pi_{I_1}) = \sign (\pi_{I_2})$, we can express the level 4 coordinates $\psi_J$ with $J \in \binom{[n]\setminus [4]}{4}$ as well as the quadrics $P_J$ in the coordinate ring $\mathbb C[\mathcal{H}]=\mathbb C[\bigwedge^4\mathbb C^n]$ as $P_J(\psi)$. Namely, $P_J(\psi)$ are of the form
\begin{small}
\begin{equation}\label{eq:PJ}
\begin{gathered}
\psi_{14j_3j_4} \psi_{23 j_1 j_2} 
- \psi_{14 j_2 j_4}\psi_{23 j_1 j_3} 
+ \psi_{14 j_2 j_3}\psi_{23 j_1 j_4} 
+ \psi_{14 j_1 j_4}\psi_{23 j_2 j_3}
- \psi_{14 j_1 j_3}\psi_{23 j_2 j_4}
+ \psi_{14 j_1 j_2}\psi_{23 j_3 j_4} \\
- \psi_{13 j_3 j_4}\psi_{24 j_1 j_2}
+ \psi_{13 j_2 j_4}\psi_{24 j_1 j_3}
- \psi_{13 j_2 j_3}\psi_{24 j_1 j_4}
- \psi_{13 j_1 j_4}\psi_{24 j_2 j_3}
+ \psi_{13 j_1 j_3}\psi_{24 j_2 j_4}
- \psi_{13 j_1 j_2}\psi_{24 j_3 j_4} \\
+ \psi_{12 j_3 j_4}\psi_{34 j_1 j_2}
- \psi_{12 j_2 j_4}\psi_{34 j_1 j_3}
+ \psi_{12 j_2 j_3}\psi_{34 j_1 j_4}
+ \psi_{12 j_1 j_4}\psi_{34 j_2 j_3}
- \psi_{12 j_1 j_3}\psi_{34 j_2 j_4}
+ \psi_{12 j_1 j_2}\psi_{34 j_3 j_4}.
\end{gathered}
\end{equation}
\end{small}

\noindent The quadrics $P_J$ \eqref{eq:PJ} for $J \in \binom{[n] \setminus [4]}{4}$ are the defining polynomials of the affine truncation variety $V_{\{2\}}^{\mathbb A}$. Therefore, we have the straightforward description of $V_{\{2\}}^{\mathbb A}$ as follows.

\begin{proposition}\label{prop:defeqs}
The affine truncation variety $V_{\{2\}}^{\mathbb A} \subset \mathcal{H}\simeq \mathbb C^{{n \choose 4}}$ is isomorphic to
the affine graph of the map
$$
\CC^{6 \cdot \binom{n - 4}{2}} \to \CC^{\binom{n - 4}{4}}, \quad \psi \mapsto (P_J(\psi))_{J \in {{[n]\setminus [4]} \choose 4}}.
$$
\end{proposition}

\indent In the above proposition, we have an isomorphism and not an equality since $V_{\{2\}}^{\mathbb A}$ is defined in a bigger ambient space, and it is also cut out by the linear equations imposed by level $1$ and level $3$ coordinates.
Homogenizing the corresponding ideal with respect to $\psi_{1234}$ yields the homogeneous prime ideal for the projective truncation variety $V_{\{2\}} \subset \mathbb P(\mathcal{H})$ as in~\cite[Theorem 3.2]{faulstich2024algebraic}. More precisely,
\begin{align*}
\mathcal I(V_{\{2\}}) &= \overline{\langle \{\psi_I \mid I \text{ of level }1,3\} \cup \{ - \psi_J + P_J \mid J \text{ of level }4\} \rangle} \\
&=
\langle \{\psi_I \mid I \text{ of level }1,3\} \cup \{ - \psi_J \psi_{[4]} + P_J \mid J \text{ of level }4\} \rangle: \langle\psi_{[4]}\rangle^{\infty}\subseteq 
\mathbb C[\textstyle{\bigwedge^4 \mathbb C^n}]
\end{align*}
where the bar denotes the homogenization of the ideal. 
For $J \in {{[n]\setminus [4]} \choose 4}$ of level $4$ we denote by 
\begin{align}\label{homogenized quadrics}
Q_{J}(\psi) := -\psi_{J}\psi_{[4]} + P_{J}(\psi)\in \mathbb C[\textstyle{\bigwedge^4 \mathbb C^n}],
\end{align}
the homogenized quadrics.

\begin{example}[$8$ and $9$ spin orbitals]\label{ex: 8 and 9 spin orbitals}
For $n=8$, there is only one level $4$ index, namely, $\{5,6,7,8\}$, hence, $\mathcal I(V_{\{2\}}) = \langle \{\psi_I \mid I \text{ of level }1,3\} \cup \{Q_{5678}\} \rangle$. Thus, the truncation variety $V_{\{2\}}$ has degree 2. 
For $n=9$, there are five level $4$ indices. One can computationally verify that the truncation variety in this case is a complete intersection, defined by the linear equations and the five homogeneous quadrics $Q_{5678}$, $Q_{5679}$, $Q_{5689}$, $Q_{5789}$ and $Q_{6789}$. We note that the saturation is not needed in this case, and that the degree of the truncation variety is $2^5 = 32$.
\end{example}

\noindent In practice, this is hard to generalize computationally: the saturation is very expensive, since it includes the computation of a Gröbner basis with doubly exponential worst-case complexity. Note that for $n = 10$ we have $15$ homogenized quadrics in $90$ variables. We ran this example on an interactive compute server at the Max Planck Institute for Mathematics in the Sciences\footnote{2× Intel Xeon Gold 6240 CPUs with 36 cores/72 threads per CPU, 1 TB DDR4 RAM, and NVMe SSD storage.} with 17 maximum allowable threats in \texttt{Macaulay2} to compute the degree of the projective truncation variety (see the code in \cite{Faulstich2025SupplementaryCode}). It did not terminate after more than two weeks. We can, however, use the structure of these quadrics to restate the problem.

\subsection{Complete intersections}\label{sec:complete intersections}
\noindent We recall that a projective variety $X \subset \mathbb P^n$ is complete intersection if it is cut out by $\codim(X)$ many equations. We say that a set of equations $\mathcal{G}=\{g_1, \ldots , g_r\}$ is complete intersection if the zero locus $V(\mathcal{G})$ has codimension $r$. Also, in this case we say that $V(\mathcal{G})$ is a complete intersection with respect to $\mathcal{G}$. 

Let $\mathbb C^n = U \oplus W$ where $U=\text{Span}(e_1,\ldots, e_4)$ and $W=\text{Span}(e_5,\ldots , e_n)$. 
One can decompose the ambient space $\mathcal{H}=\bigwedge^4\mathbb C^n$ as
\begin{equation*}
\textstyle{\bigwedge^4(U\oplus W) = \bigwedge^4U \ \oplus \ \big(\bigwedge^3U \otimes W \big)\ \oplus \ \big( \bigwedge^2U \otimes \bigwedge^2 W \big) \ \oplus \ \big( U \otimes \bigwedge^3W \big) \oplus \bigwedge^4W} .
\end{equation*}
The linear forms $\{\psi_I \mid I \text{ of level }1,3\} \subset \mathbb C[\bigwedge^4\mathbb C^n]$ correspond to the coordinates of the summands $\bigwedge^3U \otimes W$ and $U \otimes \bigwedge^3W$. In light of the containment $V_{\{2\}}\subset \textrm{V}(\psi_I \mid I \text{ of level }1,3)$, we consider the truncation variety in the ambient space in which it is non-degenerate, denoted by
\[ \mathscr{V}_{\{2\}} \subset \mathbb P\left(\textstyle{\bigwedge^4U \oplus \big( \bigwedge^2U \otimes \bigwedge^2 W \big) \oplus \bigwedge^4W}\right)  .\]

\noindent Note that ${\codim \mathscr{V}_{\{2\}} = {{n-4} \choose 4}}$. Intersecting $\mathscr{V}_{\{2\}}$ with the affine chart $\{\psi_{[4]}\neq 0\}$, we obtain the affine variety 
\begin{equation*}
\mathscr{V}_{\{2\}}^{\mathbb A} := \textrm{V}(\psi_J - P_J \mid J \textrm{ of level 4}) \subset \textstyle{\left( \bigwedge^2 U \otimes \bigwedge^2 W \right) \oplus \bigwedge^4 W}.
\end{equation*}

\noindent Next, consider the naive homogenization of $\mathscr{V}_{\{2\}}^{\mathbb A}$
\begin{equation*}
V_Q := \textrm{V}(Q_J \mid J \textrm{ of level }4) \subset \mathbb P \textstyle{\left(\bigwedge^4U \oplus\left( \bigwedge^2 U \otimes \bigwedge^2 W \right) \oplus \bigwedge^4 W \right)},
\end{equation*}
obtained by homogenizing the generators of the ideal defining $\mathscr{V}_{\{2\}}^{\mathbb A}$. We note that the non-degenerate truncation variety $\mathscr{V}_{\{2\}}$ is an irreducible component of $V_Q$. More precisely, $\mathscr{V}_{\{2\}}$ is obtained from the saturation
$\mathcal I(\mathscr{V}_{\{2\}}) = \langle Q_J \mid J \textrm{ of level }4 \rangle : \langle \psi_{[4]}\rangle^{\infty}$, i.e., $\mathscr{V}_{\{2\}}$ arises from $V_Q$ by removing all components contained entirely in the hyperplane 
\begin{equation*}
H_{\infty} := \textrm{V}(\psi_{[4]}) \subset \mathbb P \textstyle{\left(\bigwedge^4U \oplus\left( \bigwedge^2 U \otimes \bigwedge^2 W \right) \oplus \bigwedge^4 W \right)}.
\end{equation*}
The intersection $V_Q \cap H_\infty$ is isomorphic to
\begin{equation*}
V_P:= \textrm{V}(P_J \mid J \textrm{ of level }4) \subset 
\textstyle{\mathbb P\big(\big( \bigwedge^2U \otimes \bigwedge^2 W \big) \oplus \bigwedge^4W \big)}. 
\end{equation*}

\begin{lemma}\label{lemma:transint}
The quadrics $Q_J$ with $J$ of level 4  intersect transversally on the chart $\{\psi_{[4]} \neq 0\}$.
\begin{proof}
By \cite[Section 2.2.1]{shafarevich2013basic}, it suffices to show that
\begin{equation*}
\codim \bigcap\limits_{J \text{ of level 4}} T_p(Q_J) 
= \sum\limits_{J \text{ of level 4}} \codim \, \textrm{V}(Q_J) = {{n-4}\choose 4}
\end{equation*}
for $p \in V_Q \setminus \textrm{V}(\psi_{[4]} )$.
Denote $\Jac_J$ the Jacobian of $\textrm{V}(Q_J)$. Note that 
\begin{equation*}
(\Jac_J(p))_I = \left\{ 
\begin{array}{ll}
    -p_J & {\rm if}~I=[4], \\
     \pm p_{([4] \cup J) \setminus I} & {\rm if}~I \textrm{ of level 2 s.t. } I \cap ([n] \setminus [4]) \subset J \cap ([n] \setminus [4]))\\
     -p_{1234} & {\rm if}~I = J, \\
     0 & \textrm{otherwise,}
\end{array}
\right.
\end{equation*}
and therefore
\begin{equation*}
\Jac_J(p) \cdot v
=
-p_{1234} v_J - p_J v_{1234} + \varphi_J(p,v),
\end{equation*}
where $\varphi_J(p,v)$ is a polynomial using $p_I$ and $v_I$ for $I$ of level 2. Then, 
\begin{equation*}
T_p(\textrm{V}(Q_J))  = \ker \Jac_J(p) = \left\{ v \mid v_{J} = \frac{\varphi_J(p,v)- p_J v_{1234}}{p_{1234}} \right\}.
\end{equation*}
Hence, all variables except for $v_J$ are free in $T_p(\textrm{V}(Q_J))$. However, $v_J$ is free in all the other tangent spaces $T_p(\textrm{V}(Q_{J'}))$, in fact, the index $J$ does not occur in $Q_{J'}$. As a consequence, $\bigcap_{J} T_p(Q_J) $ is a complete intersection and has the desired codimension.
\end{proof}
\end{lemma}

\begin{proposition} \label{ci iff sat}
    The variety $V_P$ is a complete intersection if and only if $\mathscr{V}_{\{2\}} = V_Q$.
    In this case, the truncation variety $V_{\{2\}} \subset \mathbb P^{{n \choose 4}-1}$ coincides with the naive homogenization of the affine variety $V_{\{2\}}^{\mathbb A} \subset \mathbb C^{n \choose 4}$, that is
    \[ V_{\{2\}} = {\normalfont \textrm{V}(Q_J \mid J \textrm{ of level }4) \cap \textrm{V}(\psi_I \mid I \textrm{ of level }1,3)} .\]
    \begin{proof}
        If the quadrics $P_J$ are a complete intersection, then $$\codim (V_Q \cap H_\infty) = \codim (V_P)+1 = {{n-4} \choose 4}+1 = \codim (\mathscr{V}_{\{2\}})+1.$$
        Since $V_Q = \mathscr{V}_{\{2\}} \cup (V_Q \cap H_\infty)$, this implies that $\dim V_Q = \dim \mathscr{V}_{\{2\}}$. Also, by \cite[Corollary 2.1]{brysiewicz2019degree} (or \cite[Proposition 16]{brysiewicz2023lawrence}), then $\deg \mathscr{V}_{\{2\}} = \deg \mathscr{V}_{\{2\}}^{\mathbb A} = \deg V_Q$. 
        Thus, since $\mathscr{V}_{\{2\}} \subset V_Q$ and since $V_Q$ is locally a complete intersection in the chart $\{\psi_{[4]}\neq 0\}$, equality holds.
        Assume the quadrics $P_J$ are not a complete intersection. Then, $\dim (V_Q \cap H_\infty) \geq \dim \mathscr{V}_{\{2\}}$ and therefore $(V_Q \cap H_\infty) \not\subset \mathscr{V}_{\{2\}}$: equality is not possible since $\mathscr{V}_{\{2\}}$ may only properly intersect (in particular, cannot be contained in) $H_\infty$ by definition.
        The naive homogenization $V_Q$ thus has at least two components, one of them being $\mathscr{V}_{\{2\}}$.
    \end{proof}
\end{proposition}

\noindent If the quadrics $Q_J$ are the only non-linear generators of the homogeneous prime ideal of the truncation variety, then the ideal $\langle Q_J \ | \ J \textrm{ of level 4} \rangle$ is prime. We may use this fact to compute its degree.

\begin{lemma}\label{lemma: gen transint}
If $\mathcal I(V_Q)$ is prime, then the quadrics $Q_J$ intersect generically transversally. 

\begin{proof}
Any point $p \in V_Q \cap \{\psi_{[4]} \neq 0 \}$ is a generic point of $V_Q$ at which by Lemma~\ref{lemma:transint}, the quadrics $Q_J$ intersect transversally.
\end{proof}
\end{lemma}

\begin{corollary}\label{cor: degree truncation}
    If the variety $V_P \subset \mathbb P^{6{n-4 \choose 2} + {n-4 \choose 4}-1}$ is a complete intersection with respect to the $P_J$, then the degree of the truncation variety $V_{\{2\}} \subset \mathbb P^{{n \choose 4}-1}$ is $2^{{n-4} \choose 4}$.

    \begin{proof} This follows from Lemma~\ref{lemma: gen transint}, since by
        \cite[Theorem 18.3]{harris1992algebraic}, the degree of a variety with generators intersecting generically transversally is the product of the generators' degrees.
    \end{proof}
\end{corollary}

\subsection{Finding bounds}\label{sec: numerical evidence}

In order to apply Corollary~\ref{cor: degree truncation}, we need to characterize up to which $n$ the quadrics $P_J$ form a complete intersection.
The defining quadrics of $V_P$ for $n$ spin-orbitals form a subset of those defining $V_P$ for any larger number of spin-orbitals. In particular, if the quadrics $P_J$ in \eqref{eq:PJ} are a complete intersection for some $n$, then they are a complete intersection for all smaller $n$ as well. On the other hand, if they are not a complete intersection for some $n$, then they are not a complete intersection for all larger $n$. 
Clearly, an upper bound is given by $n=15$, since for $n \geq 15$ there are more quadrics $P_J \in \mathbb C[\bigwedge^2 U \otimes \bigwedge^2 W]$ than variables in $\mathbb C[\bigwedge^2 U \otimes \bigwedge^2 W]$. Thus, the dimension of $V_P$ cut out by the quadrics must be higher than $6 \cdot {{n-4} \choose 2} - {{n-4} \choose 4}$. In actuality, the bound is even lower.

\begin{proposition}\label{prop: n13}
    If $n \geq 13$, then the variety $V_P$ is not a complete intersection with respect to the $P_J$.
    \begin{proof}
        Consider the ideal $L = \langle \psi_{12I}, \psi_{13I}, \psi_{14I} \mid I \in {[n] \setminus [4] \choose 2} \rangle$ in the same ambient space as $\mathcal I(V_P)$. Since every quadric $P_J$ is contained in $L$, we have a containment of zero loci $V(L) \subset V_P$. The variety $V(L)$ is a linear space of codimension $3 {n-4 \choose 2}$. Now, if the quadrics $P_J$ were a complete intersection, we would have $\codim(V_P) = {n-4 \choose 4}$, which is larger than $3 {n-4 \choose 2}$ for $n \geq 13$, contradicting the containment.
    \end{proof}
\end{proposition}

The bound $n=13$ appears naturally in this proof. Of course, we cannot deduce from this that for $n \leq 12$ the variety $V_P$ is a complete intersection with respect to the $P_J$. However, the following numerical experiment gives some better insights. Note that, if $V_P$ is a complete intersection for $n=12$, its degree is $2^{70}$. Hence, actually proving this numerically is infeasible. 

\subsubsection*{Numerical experiment} By definition of the $P_J$'s, the variety $V_P$ is a cone over the zero locus of the $P_J$'s inside $\mathbb P(\bigwedge^2U \otimes \bigwedge^2 W)$, it is therefore enough to show that the latter locus is a complete intersection. Our strategy is as follows: First, we compute the intersection of $V_P$ with a random linear subspace of dimension ${{n-4} \choose 4}-1$, i.e., one less than the expected codimension of $V_P$. Finding a point in the intersection is a strong indicator for $V_P$ having codimension smaller than ${{n-4} \choose 4}$. Finding no point would suggest $V_P$ is a complete intersection. 
We implement this using the \verb|solve| function from \verb|HomotopyContinuation.jl| (see \cite{Faulstich2025SupplementaryCode}). 

We start with $n=13$. Although, we already know the outcome in this case by Proposition~\ref{prop: n13}, but we do so to later compare the result (in particular the computation time) with the one from the computations for $n=12$.
We intersect $V_P$ with a randomly generated linear subspace of dimension $\binom{13-4}{4} - 1 = 125$. To solve the resulting polynomial system without computing the full zero-dimensional solution set, we utilize homotopy continuation iterators from \cite{breiding2025homotopyiterators}. This allows us to extract the first valid solution from each system efficiently.
A point in the intersection is found within several minutes on a server\footnote{This was performed in a server with the following hardware: 4x 18-Core Intel Xeon E7-8867 v4 at 2.4 GHz, 3072 GB RAM.}. 
This is a definitive indicator that the actual codimension of $V_P$ is strictly lower than the number of generators, as otherwise the intersection would be empty.
We did ensure that the computed points were valid solutions to the system by requiring the residual norm to be below a tolerance of $10^{-9}$.
We additionally computed the rank of the Jacobian at the points that we found. For a smooth point, this coincides with the codimension of the variety. 
The rank of the Jacobian matrix at these points was then computed using the standard Singular Value Decomposition (SVD) routine provided by \texttt{LinearAlgebra} in Julia. This routine determines the numerical rank based on machine precision ($\approx 10^{-16}$) relative to the largest singular value. The expected codimension for $n=13$ is $126$, yet the observed numerical rank consistently (for 300 samples) falls between $110$ and $114$. This deficit of $12$-$16$ dimensions is far larger than the ambiguity introduced by the choice of numerical tolerance (which typically only affects singular values near the noise floor).

For $n=12$, we repeated the procedure above and intersected with a linear subspace of dimension ${{12-4} \choose 4}-1 = 69$. In contrast to $n=13$, the computations for a single solution did not terminate within two months\footnote{The same server as above.}. Note that for $n=13$, the system is much larger but still a solution was found within several minutes. This suggests that there does not exist a point in the intersection and that $V_P$ indeed does not have codimension smaller than $\binom{12-4}{4}$, although it is not strong numerical evidence. Additionally, one can check the Jabobian criterion by sampling a point via intersecting with a linear subspace of dimension ${{12-4} \choose 4} = 70$, so that the intersection has expected dimension zero. Again, we obtain the first solution using a homotopy iterator (in this case non-singular) and evaluate the Jacobian of $V_P$. For $n=12$, we generate $500$ samples and observe that the rank of the Jacobian was always $70$, which coincides with the expected codimension ${{12-4} \choose 4}$. Note however, that \verb|HomotopyContinuation.solve| is heavily biased towards isolated points, thus it is possible that the points that we find do not lie on the highest dimensional components. Still, overall, this together with the other computational results compared to the $n=13$ case provides numerical evidence for the following conjecture.

\begin{conjecture}
   The variety $V_P$ is a complete intersection with respect to the $P_J$ if and only if $n \leq 12$. Equivalently, the truncation variety $V_{\{2\}}$ coincides with the naive homogenization of the affine truncation variety, i.e.
    \[ V_{\{2\}} = \textrm{V}(Q_J \mid J \textrm{ of level }4) \cap \textrm{V}(\psi_I\mid I \textrm{ of level }1,3),  \]
    if and only if $n \leq 12$. 
\end{conjecture}

\section{Quadrics $P_J$ and Pfaffians}\label{sec: quadrics PJs and pfaffians}

We aim to have a deeper understanding of the geometry of the CDD truncation varieties $V_{\{2\}}$ for $d=4$. In particular, we focus on the variety $V_P$ defined by the $\binom{n-4}{4}$ quadrics $P_J$ as $J $ varies in $\binom{[n]\setminus[4]}{4}$. Using techniques from representation theory, we unveil a Pfaffian structure inside such equations. For every $J \in \binom{[n] \setminus [4]}{4}$, we can write the quadric $P_J$ in \eqref{eq:PJ} as
\[
P_J(\psi) = \sum_{I \subseteq \binom{J}{2}} \sign(I)(\psi_{12I}\psi_{34I^c} - \psi_{13I}\psi_{24I^c} + \psi_{14I}\psi_{23I^c}),
\]
where $I^c := J\setminus I$ and $\sign(I)$ is the sign of the permutation $J \mapsto (I, I^c)$. In the notation of Section~\ref{sec:complete intersections}, the $P_J$'s are polynomials in the coordinate ring $\mathbb C[\bigwedge^2U \otimes \bigwedge^2W]$, and being quadrics they actually lie in the degree-$2$ component $S^2(\bigwedge^2U \otimes \bigwedge^2W)^\vee$. \\
\hfill\break
\indent The space $S^2(\bigwedge^2U \otimes \bigwedge^2W)^\vee$ is a representation of $\GL(U)\times \GL(W)$, and as such it decomposes into irreducible $\GL(U)\times \GL(W)$-representations. Alternatively, one can look at it as a representation of $\SL(U)\times \SL(W)$. However, we are interested in the action of the subgroup $\mathfrak{S}_{[4]}\times \mathfrak{S}_{[n] \setminus [4]} \subset \GL(U)\times \GL(W)$, where $\mathfrak{S}_{[4]}$ (resp. $\mathfrak{S}_{[n]\setminus [4]}$) denotes the symmetric group permuting the basis vectors of $U$ (resp. $W$). From the above expression of the $P_J$'s, one can see that the group $\mathfrak{S}_{[4]}\times \mathfrak{S}_{[n]\setminus [4]}$ acts on the $P_J$'s as follows:
\begin{equation}\label{eq:action of S4 on PJ}
(\tau,id) \cdot P_J = \sign(\tau)P_J \ \ \ \ \text{and} \ \ \ \ (id,\sigma) \cdot P_J = \pm P_{\sigma(J)}, 
\end{equation}
where the sign of the action of $(id,\sigma)$ depends on which indices of $J$ are involved from the permutation $\sigma$. 

\begin{example}\label{example:action of SJ on PJ}
For $\sigma \in \mathfrak{S}_{J}\subset \mathfrak{S}_{[n]\setminus [4]}$ permuting the indices in $J$ and fixing the indices in $[n]\setminus ([4]\cup J)$, it holds $(id, \sigma)\cdot P_J = \sign(\sigma)P_J$. However, if $\sigma \notin \mathfrak{S}_J$, the previous equality can fail. For instance, for $i< j < k < a <b <c$ one gets
$(i \, j)\cdot P_{iabc} = P_{jabc} = (i \, j \, k)\cdot P_{iabc}$. Similarly, for $i<j<k<h<a<b$ it holds $(i \, j \, k) \cdot P_{ijab} = P_{jkab} = (i \, j \, k \, h) \cdot P_{ijab}$.
\end{example}

\noindent In particular, the linear span $\langle P_J|J\rangle_\mathbb C$ is $\mathfrak{S}_4\times \mathfrak{S}_{n-4}$-invariant. On the other hand, the set of generators $\{P_J|J\}$ is invariant under the action of $\mathfrak{A}_4\times \mathfrak{A}_{n-4}\subset \SL(U)\times \SL(W)$, where $\mathfrak{A}_4=\mathfrak{S}_4\cap \SL(U)$ is the alternating group of even permutations (similarly for $\mathfrak{A}_{n-4}$). \\
\indent Throughout this section, our strategy is as follows: We begin by decomposing the space $S^2(\bigwedge^2U \otimes \bigwedge^2W)^\vee$ into irreducible representations of $\SL(U)\times \SL(W)$, then we look at each summand as representation of $\mathfrak{S}_4\times \mathfrak{S}_{n-4}$, and finally check for which of them the linear span $\langle P_J|J \rangle_{\mathbb C}$ is a $\mathfrak{S}_4\times \mathfrak{S}_{n-4}$-subrepresentation. We will make extended use of the theory of representations of the groups $\GL(\mathbb C^N)$ and $\mathfrak{S}_4$, of which we recall some basic facts below. We will use the terms {\em representations} and {\em modules} interchangeably.

\subsection{Preliminaries on Pfaffians, $\SL$-modules and $\mathfrak{S}_4$-modules}\label{subsec:preliminaries modules}

\indent The Pfaffian of a $4\times 4$ skewsymmetric matrix $M\in \bigwedge^2\mathbb C^4$ is (see \cite[Section 2.7.4]{landsberg2011tensors})
\[ \Pf {\footnotesize \begin{pmatrix} 0 & a & b & c \\ -a & 0 & d & e \\ -b & -e & 0 & f \\ -c & -e & -f & 0 \end{pmatrix}} = af - be + cd .\]
In particular, it is a quadratic polynomial in $S^2(\bigwedge^2\mathbb C^4)^\vee$. As a tensor, it can be written as 
\begin{equation}\label{eq:Pfaffian as tensor} 
    \Pf = (x_1\wedge x_2)(x_3\wedge x_4) - (x_1\wedge x_3)(x_2\wedge x_4) + (x_1\wedge x_4)(x_2\wedge x_3) .
    \end{equation}

\medskip

\indent Irreducible representations of $\SL(\mathbb C^N)$ are parametrized by integer tuples $\mu \in \mathbb Z^{N-1}$ called {\em highest weights}. We denote them by $V_{\mu}^{A_{N-1}}$, where $A_{N-1}$ is the Dynkin type of $\SL(\mathbb C^N)$. If $\mu = (a_1,\ldots , a_{N-1})$, then the irreducible representation $V_{\mu}^{A_{N-1}}$ lies in the tensor product $S^{a_1}\mathbb C^N \otimes S^{a_2}(\bigwedge^2\mathbb C^N) \otimes \ldots \otimes S^{a_{N-1}}(\bigwedge^{N-1}\mathbb C^n)$ (see \cite[Section 15.3]{fulton2013representation}). Observe that $S^{a_1}\mathbb C^N=V_{(a_1,0,\ldots, 0)}^{A_{N-1}}$ and $\bigwedge^i\mathbb C^{N}=V_{(0,\ldots,1_i,\ldots 0)}^{A_{N-1}}$ are irreducible, while $S^{a_i}(\bigwedge^i\mathbb C^N)$ is not irreducible for every $i\geq 2$ and $a_i \geq 2$.\\
For $R_1$ and $R_2$ being $\SL(\mathbb C^{N_1})$- and $\SL(\mathbb C^{N_2})$-modules respectively, the tensor product $R_1\otimes R_2$ is a representation of $\SL(\mathbb C^{N_1})\times \SL(\mathbb C^{N_1})$. We recall that, as $\SL(\mathbb C^{N_1})\times \SL(\mathbb C^{N_1})$-module, $S^2(R_1\otimes R_2)$ decomposes into \cite[Section 6.5.2]{landsberg2011tensors}
\begin{equation}\label{eq:S2(A*B)}
S^2(R_1\otimes R_2) = (S^2R_1 \otimes S^2R_2) \oplus ({\textstyle \bigwedge^2R_1 \otimes \bigwedge^2R_2}) .
\end{equation}
If $R_1$ and $R_2$ irreducible, the above decomposition is into irreducible components as well.

\medskip

\indent Regarding representations of $\mathfrak{S}_4$, we refer to \cite[Section 2.3]{fulton2013representation} for details. There are five irreducible representations of $\mathfrak{S}_4$, parametrized by the partitions of $4$. They are:
\begin{itemize}
\item the $1$-dimensional trivial representation $[4]$;
\item the $3$-dimensional standard representation $[3,1]$;
\item the $2$-dimensional representation $[2,2]$;
\item the $3$-dimensional representation $[2,1,1]$;
\item the $1$-dimensional sign-representation $[1^4]$.
\end{itemize}
\noindent The order in which they appear above reflects the duality given by tensoring each module by the sign-representation: for instance, $[3,1] \otimes [1^4]=[2,1,1]$. Furthermore, the trivial representation $[4]$ is the unit element with respect to the tensor product, i.e., $[4]\otimes [\lambda] = [\lambda]$ for any partition $\lambda \vdash 4$. We will use the following well-known identities:
\begin{align}
{\textstyle \bigwedge^k}[3,1] & = [4-k,1^k] & \forall k=0,\ldots, 3 \label{eq:wedge[3,1]} \\
[3,1] \otimes [3,1] & = [4] \oplus [3,1] \oplus [2,2] \oplus [2,1,1] .\label{eq:[3,1]*[3,1]}
\end{align}
Observe that, since $[2,1,1] = [3,1] \otimes [1^4]$, the decomposition of $[2,1,1]\otimes [2,1,1]$ is the same as in \eqref{eq:[3,1]*[3,1]}. Finally, combining the equations \eqref{eq:wedge[3,1]} and \eqref{eq:[3,1]*[3,1]} yields
\begin{align}\label{eq:sym[3,1]}
S^2[3,1] & = [3,1]\otimes [3,1] - {\textstyle \bigwedge^2}[3,1] = [4] \oplus [3,1] \oplus [2,2] \nonumber \\
S^2[2,1,1] & = [4] \oplus [2,1,1] \oplus [2,2] .
\end{align}

\subsection{Decomposition of the second symmetric power}\label{subsec: decomposition}

\indent Now, we consider the decomposition of $S^2(\bigwedge^2U \otimes \bigwedge^2W)^\vee$ into irreducible $\SL(U)\times \SL(W)$-submodules. First, we decompose it as
\begin{align*}
\textstyle{S^2(\bigwedge^2U \otimes \bigwedge^2W)^\vee} & \stackrel{\eqref{eq:S2(A*B)}}{=}  \big[\underbrace{\textstyle{S^2(\bigwedge^2U)^\vee \otimes S^2(\bigwedge^2W)^\vee}}_{=: A}\big] \ \oplus \ \big[\underbrace{\textstyle{\bigwedge^2(\bigwedge^2U)^\vee \otimes \bigwedge^2(\bigwedge^2W)^\vee}}_{=: B}\big]  .
\end{align*}
\noindent Then, using the LiE package \cite{LiE} to compute the plethysms in $A,B$, for $\dim W = n-4 \geq 5$, we get
\begin{align*}
A & = \left( V^{A_3}_{(0,0,0)} \oplus V^{A_{3}}_{(0,2,0)} \right) \otimes \left( V^{A_{n-5}}_{(0,2,0, \ldots)} \oplus V^{A_{n-5}}_{(0,0,0,1,0, \ldots)} \right)  \\
B & =  V_{(1,0,1)}^{A_3}\otimes V_{(1,0,1,0, \ldots )}^{A_{n-5}}, 
\end{align*}
where in particular
\begin{itemize}
\item $V^{A_3}_{(0,0,0)}\simeq \mathbb C$ is $1$-dimensional and generated by the only $\SL(4)$-invariant of degree $2$, i.e., the Pfaffian $\Pf(\underline{e})$ of a $4\times 4$ skewsymmetric matrix $\underline{e}\in \bigwedge^2U$;
\item $V_{(1,0,1)}^{A_3}$ is the $15$-dimensional adjoint $\SL(4)$-representation $\mathfrak{sl}_4\subset U \otimes U^\vee$;
\item $V^{A_{n-5}}_{(0,0,0,1,0,\ldots)} = \bigwedge^4W = \big\langle \Pf([\underline{f}]_J) \ | \ J \in \binom{[n-4]}{4} \big\rangle_\mathbb C$ is spanned by the Pfaffians of the $4\times 4$ diagonal submatrices of a $(n-4)\times (n-4)$ skewsymmetric matrix $\underline{f}\in \bigwedge^2W$.
\end{itemize}
Therefore, we get
\begin{align}\label{eq:A+B}
A\oplus B & = \left(\Pf(\underline{e}) \otimes V^{A_{n-5}}_{(0,2,0, \ldots)}\right) \ \oplus \ \left( \Pf(\underline{e}) \otimes \textstyle{\bigwedge^4 W} \right) \ \nonumber \\
& \ \ \ \oplus \left( V^{A_{3}}_{(0,2,0)} \otimes V^{A_{n-5}}_{(0,2,0, \ldots)} \right) \ \oplus \ \left( V^{A_{3}}_{(0,2,0)} \otimes \textstyle{\bigwedge^4 W} \right) \ \oplus \ \left( \mathfrak{sl}_4 \otimes V_{(1,0,1,0, \ldots )}^{A_{n-5}} \right) .
\end{align}
\indent First, we focus on the action of $\mathfrak{S}_{[4]}$ on the first tensor-entry of each summand. As $\mathfrak{S}_{[4]}$-representation, straightforward computations using identities from Section~\ref{subsec:preliminaries modules} show that the space $S^2(\bigwedge^2U)^\vee$  decomposes into
\[ \textstyle{S^2(\bigwedge^2U)^\vee} = 2  [4] \oplus 2 [3,1] \oplus 3  [2,2] \oplus 2 [2,1,1] \oplus [1^4] .\]
\noindent In particular, it contains a unique copy of the sign-representation $[1^4]$, which is generated by the Pfaffian $\Pf(\underline{e})$, so it coincides with the space $V_{(0,0,0)}^{A_3}$. Therefore, the $\SL(4)$-component $V_{(0,2,0)}^{A_3}$ does not contain the sign-representation. On the other hand, the component $V_{(1,0,1)}^{A_3}\simeq \mathfrak{sl}_4$ does not contain $[1^4]$ either, since as $\mathfrak{S}_4$-representation it decomposes into
\[ \mathfrak{sl}_4 = [4] \oplus 3[3,1] \oplus [2,2] \oplus [2,1,1] .\]
\noindent From \eqref{eq:action of S4 on PJ} we know that for every $J\in \binom{[n]\setminus[4]}{4}$, the $1$-dimensional span $\langle P_J \rangle_\mathbb C$ is isomorphic to the sign-representation $[1^4]$ as $\mathfrak{S}_{[4]}$-module. We deduce that the following containment has to hold:
\begin{equation}\label{eq:PJ in first two summands}
\langle P_J|J\rangle_{\mathbb C}\subset \Pf(\underline{e})\otimes \left( V_{(0,2,0,\ldots)}^{A_{n-5}}  \oplus  \textstyle{\bigwedge^4W}\right) .
\end{equation}

\indent Second, we aim to use similar arguments as above to exclude one of the two remaining summands. However, computing representations of $\mathfrak{S}_{n-4}$ for $n\geq 9$ is quite involved. Therefore, we look at the linear spaces $\langle P_J|J\rangle_\mathbb C$, $V_{(0,2,0,\ldots)}^{A_{n-5}}$ and $\bigwedge^4W$ as representations of $\mathfrak{S}_4 \simeq \mathfrak{S}_{J_0} \subset \mathfrak{S}_{[n]\setminus [4]}$ for a fixed subset $J_0\in \binom{[n]\setminus [4]}{4}$. More precisely, we aim to show that, as $\mathfrak{S}_4$-representations, the space $V_{(0,2,0,\ldots)}^{A_{n-5}}$ does not contain any sign-representation, while $\langle P_J|J\rangle_{\mathbb C}$ and $\bigwedge^4W$ contain the sign-representation with the same multiplicity. This will allow us to conclude in which summand the span of the $P_J$'s lies. For simplicity, we denote $m:= n-4 = \dim W$.

\begin{proposition}\label{prop:sign-rep in wedge4}
Consider the action of $\mathfrak{S}_4$ on $W = M \oplus N$ that fixes $M\simeq \mathbb C^{m-4}$ and permutes the basis vectors of $N\simeq \mathbb C^4$. Then the $\mathfrak{S}_4$-representation $\bigwedge^4W$ contains exactly $m-3$ copies of the sign-representation $[1^4]$.
\end{proposition}
\begin{proof}
As $\mathfrak{S}_4$-modules, we have the decompositions $M = \sum_{\ell =1}^{m-4}[4]_\ell$ and $N=[4]_0 \oplus [3,1]_0$, where $[4]_\ell, [4]_0$ are trivial representations and  $[3,1]_0$ is the standard representation of dimension $3$. For simplicity, we use the compact expression $\mathbb C^{m-3} = [4]_0 \oplus \sum_{\ell =1}^{m-4}[4]_\ell$. Applying the fourth exterior power we get
{\small
\begin{align*}  \textstyle{\bigwedge^4W} & \ = {\textstyle \bigwedge^4}\left(\mathbb C^{m-3} \oplus [3,1]_0 \right) \\
& \ =  {\textstyle \bigwedge^4}\mathbb C^{m-3} \oplus \left({\textstyle \bigwedge^3} \mathbb C^{m-3} \wedge [3,1]_0 \right) \oplus \left( {\textstyle \bigwedge^2} \mathbb C^{m-3} \wedge {\textstyle \bigwedge^2}[3,1]_0 \right) \oplus \left( \mathbb C^{m-3} \wedge {\textstyle \bigwedge^3}[3,1]_0 \right) \oplus {\textstyle \bigwedge^4}[3,1]_0  \\
& \stackrel{\eqref{eq:wedge[3,1]}}{=} {\textstyle \binom{m-3}{4}[4] \oplus \binom{m-3}{3}[3,1] \oplus \binom{m-3}{2}[2,1,1] \oplus (m-3)[1^4]} .
\end{align*}
}
\end{proof}

\begin{proposition}\label{prop:sign-rep in sym2wedge2}
Consider the action of $\mathfrak{S}_4$ on $W = F \oplus G$ that fixes $F\simeq \mathbb C^{m-4}$ and permutes the basis vectors of $G\simeq \mathbb C^4$. Then, as $\mathfrak{S}_4$-representation, the linear space $V_{(0,2,0,\ldots)}^{A_{n-5}}$ contains no copy of the sign-representation $[1^4]$.
\end{proposition}
\begin{proof}
The $\GL(W)$-representation $V_{(0,2,0,\ldots)}^{A_{m-1}}$ is the kernel of the $\GL(W)$-equivariant map  
\[ \begin{matrix}
S^2(\textstyle{\bigwedge^2W}) & \rightarrow & \textstyle{\bigwedge^4W} \\
\alpha\cdot \beta & \mapsto & \alpha \wedge \beta
\end{matrix} \ .\] 
Therefore, to prove the hypothesis, it is enough to show that the module $S^2(\bigwedge^2W)$ contains as many copies of $[1^4]$ as $\bigwedge^4W$. Again, we decompose $W=\sum_{\ell=1}^{m-4}[4]_\ell \oplus [4]_0 \oplus [3,1]_0= \mathbb C^{m-3}\oplus [3,1]$. \\
First, we compute the $2$nd exterior power of $W$.
\begin{align*}
{\textstyle \bigwedge^2}W & \  = {\textstyle \bigwedge^2}\left( \mathbb C^{m-3} \oplus [3,1]\right)  = {\textstyle \bigwedge^2}\mathbb C^{m-3} \oplus \left(\mathbb C^{m-3}\wedge [3,1] \right) \oplus {\textstyle \bigwedge^2}[3,1] \\
& \stackrel{\eqref{eq:wedge[3,1]}}{=} {\textstyle \binom{m-3}{2}}[4] \oplus (m-3)[3,1] \oplus [2,1,1] .
\end{align*}
Second, applying the $2$nd symmetric power, we get
\begin{align*}
S^2({\textstyle \bigwedge^2}W) & \ = S^2\left( {\textstyle \binom{m-3}{2}}[4] \oplus (m-3)[3,1] \oplus [2,1,1] \right) \\
& \ \simeq {\textstyle \binom{m-3}{2}^2}S^2[4] \oplus {\textstyle \binom{m-3}{2}^2}(m-3)[4]\cdot [3,1] \oplus {\textstyle \binom{m-3}{2}^2}[4]\cdot [2,1,1] \oplus  (m-3)^2 S^2[3,1] \\
& \ \quad \oplus (m-3)[3,1]\cdot [2,1,1] \oplus S^2[2,1,1] \\
& \stackrel{\eqref{eq:sym[3,1]}}{=} {\textstyle \binom{m-3}{2}^2}[4] \oplus {\textstyle \binom{m-3}{2}^2}(m-3)[3,1] \oplus {\textstyle \binom{m-3}{2}^2}[2,1,1] \oplus (m-3)^2\big( [4] \oplus [3,1] \oplus [2,2] \big) \\
& \ \quad \oplus (m-3)[3,1]\cdot [2,1,1] \oplus \big( [4] \oplus [2,1,1] \oplus [2,2] \big) .
\end{align*}
Therefore, all copies of $[1^4]$ inside $S^2(\bigwedge^2W)$ lie inside the summands $(m-3)[3,1]\cdot [2,1,1]$. To decompose this product, we use that $[2,1,1]=[3,1]\otimes [1^4]$ and the decomposition of $[3,1]\otimes [3,1]$:
\begin{align*}
(m-3)[3,1]\cdot [2,1,1] & \ = (m-3) [3,1] \otimes \big([3,1] \otimes [1^4] \big) \\
& \stackrel{\eqref{eq:[3,1]*[3,1]}}{=} (m-3)\big( [4] \oplus [3,1] \oplus [2,2] \oplus [2,1,1]\big) \otimes [1^4]\\
& \ = (m-3)\big( [1^4] \oplus [2,1,1] \oplus [2,2] \oplus [3,1] \big) .
\end{align*}
We conclude that $S^2(\bigwedge^2W)$ has $m-3$ copies of $[1^4]$, all of them coming from the component $\bigwedge^4W$. In particular, $V_{(0,2,0,\ldots)}^{A_{m-1}}$ contains no copy of the sign-representation.
\end{proof}

\begin{theorem}\label{thm:Pfaffians}
In the above notation, the quadrics $P_J$'s are linear combinations of tensor-product of Pfaffians. More precisely,
\[ \langle P_J|J\rangle_\mathbb C = \Pf(\underline{e}) \otimes \left\langle \Pf([\underline{f}]_J) \mid J\right\rangle_\mathbb C .\]
\end{theorem}
\begin{proof}
Since $\langle P_J|J\rangle_\mathbb C$ is $\mathfrak{S}_{4}\times \mathfrak{S}_{[n]\setminus [4]}$-invariant, it lies in only one of the summands of \eqref{eq:PJ in first two summands} considered as $\GL(U)\times \GL(W)$-representations. In light of Proposition~\ref{prop:sign-rep in wedge4} and Proposition~\ref{prop:sign-rep in sym2wedge2}, it is enough to prove that $\langle P_J|J\rangle_\mathbb C$ as $\mathfrak{S}_4$-representation contains at least one copy of $[1^4]$ in order to conclude that it lies in (and by dimensional count, actually coincides with) $\bigwedge^4W$. This can be immediately shown by fixing $J_0 \in \binom{[n]\setminus [4]}{4}$ and observing that the line $\langle P_{J_0}\rangle_\mathbb C$ is the sign-representation for the action of $\mathfrak{S}_{J_0}\simeq \mathfrak{S}_4$. 
\end{proof}

\begin{remark}
We exhibit all the $m-3$ copies of $[1^4]$ inside $\langle P_J|J\rangle_\mathbb C$ (as many as in $\bigwedge^4W$). Without loss of generality we can assume $J_0=\{5,6,7,8\} \in \binom{[n]\setminus [4]}{4}$. Consider the action of $\mathfrak{S}_4\simeq \mathfrak{S}_{J_0} \subset \mathfrak{S}_{[n]\setminus [4]}$ permuting the first four basis vectors of $W\simeq \mathbb C^{n-4}$ and fixing the remaining $n-8$ basis vectors. Consider the decomposition
\begin{align*}
    \langle P_J|J\rangle_{\mathbb C} & = \ \langle P_{J_0}\rangle_{\mathbb C} \ \oplus \ \langle P_J \mid |J\cap J_0|=3\rangle_{\mathbb C} \ \oplus \ \langle P_J \mid |J\cap J_0|=2\rangle_{\mathbb C} \  \\
    & \quad \oplus \ \langle P_J \mid |J\cap J_0|=1\rangle_{\mathbb C} \ \oplus \ \langle P_J \mid J\cap J_0=\emptyset \rangle_{\mathbb C} .
    \end{align*}
Each summand is invariant under the action of $\mathfrak{S}_{J_0}$. We study such action case by case.
\begin{enumerate}
\item For any $\tau \in \mathfrak{S}_{J_0}$ it holds $\tau \cdot P_{J_0} = \sign(\tau)P_{J_0}$, hence the line $\langle P_{J_0}\rangle_\mathbb C$ is isomorphic to the sign-representation $[1^4]$.
\item The summand $\langle P_J \mid J\cap J_0=\emptyset \rangle_{\mathbb C}$ is isomorphic to $\binom{n-8}{4}[4]$ since $\mathfrak{S}_{J_0}$ acts trivially on it.
\item Pick $P_{iabc}\in \langle P_J \mid |J\cap J_0|=1\rangle_{\mathbb C}$ where $i\in J_0$. For every transposition $(i \, j)\in \mathfrak{S}_{J_0}$ we get $(i \, j)\cdot P_{iabc} = P_{jabc}$. We deduce that no element of $\langle P_J \mid |J\cap J_0|=1\rangle_{\mathbb C}$ can generate a sign-representation. (What goes wrong here is that $(i \, j)\cdot P_{iabc} = (i \, j \, k)\cdot P_{iabc}$ -  see Example~\ref{example:action of SJ on PJ}.) 
\item Pick $P_{ijab} \in \langle P_J \mid |J\cap J_0|=2\rangle_{\mathbb C}$ where $\{i< j \} \in J_0$. For a $4$-cycle $(i \, j \, k \, h)\in \mathfrak{S}_{J_0}$ such that $j<k$ we get $(i \, j \, k \, h)\cdot P_{ijab} = P_{jkab}$. We deduce that no element of $\langle P_J \mid |J\cap J_0|=1\rangle_{\mathbb C}$ can generate a sign-representation. (What goes wrong here is that $(i \, j \, k \, h)\cdot P_{ijab} = (i \, j \, k)\cdot P_{ijab}$ -  see Example~\ref{example:action of SJ on PJ}.)
\item Fix $a \in \{9, \ldots , n\}$ and consider $\langle P_{ijka} \mid \{i,j,k\} \subset J_0\rangle_\mathbb C$. Since these $P_J$'s depend on three indices from $J_0$, a straightforward computation shows that for any $\tau \in \mathfrak{S}_{J_0}$ it holds $\tau \cdot P_{ijka} = \sign(\tau)P_{\tau(i)\tau(j)\tau(k)a}$. Let $X \in \langle P_{ijka} \mid \{i,j,k\}\subset J_0 \rangle_{\mathbb C}$: it is of the form
\[ X= \sum_{I \in \binom{J_0}{3}}\alpha_I P_{Ia} \]
for certain $\alpha_I \in \mathbb C$. The line $\langle X \rangle_\mathbb C$ is isomorphic to the sign-representation if and only if for any $\tau \in \mathfrak{S}_{J_0}$ it holds $\tau \cdot X = \sign(\tau)X$. In particular, $\langle X\rangle_\mathbb C\simeq [1^4]$ if and only if for any $\tau \in \mathfrak{S}_{J_0}$
\begin{align*}
\sign(\tau)\sum_I\alpha_IP_{\tau(I)a} = \sign(\tau) \sum_I \alpha_I P_{Ia} & \iff \quad \sum_{I} (\alpha_{\tau^{-1}(I)}-\alpha_I)P_{Ia} = 0 \\
& \iff \quad \alpha_{\tau^{-1}(I)} = \alpha_I  \quad \forall I,
\end{align*}
where the last double implication follows from the fact that the quadrics $P_{Ia}$ are linearly independent. Since the above co-implications apply to every $\tau \in \mathfrak{S}_{J_0}$, we decuce that $\langle X \rangle_{\mathbb C}\simeq [1^4]$ if and only if $X=\alpha \sum_I P_{Ia}$ (i.e., $\alpha_I=\alpha$ for all $I$). It follows that for every $a \in \{9,\ldots, n\}$ there is one sign-representation inside each submodule $\langle P_{ijka} \mid \{i,j,k\} \subset J_0\rangle_\mathbb C \subset \langle P_J \mid |J\cap J_0|=3\rangle_\mathbb C$. Therefore, there are $n-8$ copies of $[1^4]$ inside $\langle P_J \mid |J\cap J_0|=3\rangle_\mathbb C$.
\end{enumerate}
\indent We conclude that there are $m-3$ copies of $[1^4]$ inside $\langle P_J|J\rangle_\mathbb C$ as $\mathfrak{S}_{J_0}$-representation: one from $\langle P_{J_0}\rangle_\mathbb C$, and the other $m-4$ ones from $\langle P_J \mid |J\cap J_0|=3\rangle_\mathbb C$.
\end{remark}

\subsection{Example: $4$ electrons in $8$ spin-orbitals}

We give an explicit geometric description of the $36$-dimensional truncation variety $V_{\{2\}} \subset \mathbb P\big(\bigwedge^4\mathbb C^8\big)$.\\
\indent The variety $V_{\{2\}}$ is defined by the ideal in $\mathbb C[\bigwedge^4 (\mathbb C^8)]$, i.e.,
\[ \mathcal{I}(V_{\{2\}}) = \left( \ \psi_{i j_1j_2j_3}, \ \psi_{i_1i_2i_3 j}, \ \psi_{1234}\psi_{5678} - P_{5678}(\psi) \ \right):\psi_{1234}^\infty, \]
where $i \in [4]$, $\{i_1,i_2,i_3\}\in \binom{[4]}{3}$, $j \in \{5,\ldots , 8\}$ and $\{j_1,\ldots , j_3\} \in \binom{\{5,\ldots , 8\}}{3}$. According to notation in Section~\ref{sec:complete intersections}, we write $\mathbb C^8 = U \oplus W$ where $U=\text{Span}(e_1,\ldots, e_4)$ and $W=\text{Span}(e_5,\ldots , e_8)$. Since $\psi_{ij_1j_2j_3}, \psi_{i_1i_2i_3j}\in \mathcal{I}(V_{\{2\}})$, we know that $V_{\{2\}}$ is non-degenerate in the linear space
\[ \mathbb P \left( \textstyle{ \bigwedge^4U \oplus \big( \bigwedge^2U \otimes \bigwedge^2 W \big) \oplus \bigwedge^4W} \right) \simeq \mathbb P^{37} .\]

\noindent Recall that we denote the truncation variety in this minimal embedding by $\mathscr{V}_{\{2\}}$. In particular, $\mathscr{V}_{\{2\}}$ is a hypersurface, hence saturation is not needed and $\mathcal{I}(\mathscr{V}_{\{2\}})=\mathcal{I}(V_Q)$. Therefore, it is the $36$-dimensional smooth quadric
\[ \mathscr{V}_{\{2\}} = \text{V}\big(\psi_{1234}\psi_{5678} - P_{5678}(\psi)\big) \subset \mathbb P \left( \textstyle{ \bigwedge^4U \oplus \big( \bigwedge^2U \otimes \bigwedge^2 W \big) \oplus \bigwedge^4W} \right) .\]

\indent The quadratic equation vanishes at the points $[e_{1234}] , [e_{5678}]$. However, note that any other point on the line $\ell([e_{1234}],[e_{5678}])$ does not lie in $\mathscr{V}_{\{2\}}$. Indeed, for $[\alpha:\beta]\in \mathbb P^1\setminus \{[0:1],[1:0]\}$ the evaluation of the quadric at $[\alpha e_{1234} + \beta e_{5678}]$ is $\alpha\cdot \beta \neq 0$. On the other hand, since $W$ is four-dimensional, from Theorem~\ref{thm:Pfaffians} we deduce that the $34$-dimensional linear section
\[ \mathscr{V}_{\{2\}}\cap \mathbb P\left(\textstyle{\bigwedge^2U \otimes \bigwedge^2W}\right) = \text{V}( P_{5678} ) = \text{V}\!\left(\Pf(\underline{e})\otimes \Pf(\underline{f}) \right) \]
is the zero locus of the unique invariant $\bigwedge^4U\otimes \bigwedge^4W = V_{(0,0,0)}^{A_3}\otimes V_{(0,0,0)}^{A_3}\simeq \mathbb C$ in the module $S^2\big( \bigwedge^2U \otimes \bigwedge^2W \big)^\vee$. Therefore, $\mathscr{V}_{\{2\}}$ contains the $35$-dimensional cones 
\[ \join\big([e_{1234}] , \text{V}\big( \Pf(\underline{e})\otimes \Pf(\underline{f})\big) \big) \cup \join\big([e_{5678}] , \text{V}\big( \Pf(\underline{e})\otimes \Pf(\underline{f})\big) \big) .\]
Finally, a point $z \in \mathscr{V}_{\{2\}}$ not lying on these cones is such that $P_{5678}(z) = z_{1234}z_{5678} \neq 0$. These points lie along lines spanned by one point in $\mathbb P(\bigwedge^2U\otimes \bigwedge^2W)\setminus \text{V}(\Pf(\underline{e})\otimes \Pf(\underline{f}))$ and one in $\ell([e_{1234}],[e_{5678}])\setminus \{[e_{1234}],[e_{5678}]\}$. Let $[a]$ and $[b]=[e_{1234} + b_0e_{5678}]$ be two such points respectively. A point on $\ell([a],[b])$ is of the form $[c]=[a+\lambda b]$ and it lies on $\mathscr{V}_{\{2\}}$ if and only if $\lambda^2 = b_0^{-1}P_{5678}(a)$. It follows that 
\[ \ell([a],[b]) \cap \mathscr{V}_{\{2\}} = \{ \text{pt}_1, \text{pt}_2\} .\]
One can see this also by Bèzout's theorem: either a line lies in a quadric or they intersect in two points (with multiplicity). 

\subsection{The tensor product of Pfaffians}\label{subsec:disconnected}

In the previous subsection we showed that the quadrics $P_J$ can be written as linear combinations of tensor products of Pfaffians. A natural question is whether this description can be sharpened, namely, whether each $P_J$ itself is a simple tensor product of Pfaffians. While such a factorization does not hold on the full doubles truncation variety, it does emerge in a distinguished and conceptually transparent limit that, in coupled cluster terminology, corresponds to the \emph{disconnected doubles} contribution. 

Recall that at the doubles excitation level, the exponential ansatz produces two types of terms: connected doubles, governed by genuine doubles amplitudes, and disconnected doubles, generated as products of two independent single excitations through the quadratic term in the exponential expansion. If one suppresses the connected doubles parameters and retains only singles, then the induced double excitation coefficients are given by the canonical antisymmetrized products of singles amplitudes. Algebraically, these coefficients are Pfaffians in the relevant indices. In this disconnected regime the quadrics $P_J$'s therefore factor as tensor products of Pfaffians, rather than appearing only as linear combinations. Although this regime is not intended as a practical CCD approximation, it provides a transparent limiting case in which the Pfaffian factorization is exact.

Consider once again the decomposition $V = U \oplus W$ from Section~\ref{sec:complete intersections}. For convenience, we denote $(e_1,e_2,e_3, e_4)$ a basis of $U$ and $(f_5,\ldots , f_n)$ a basis of $W$.
\noindent For any $\{i_1,i_2\}\in \binom{[4]}{2}$ and $\{j_1,j_2\}\in \binom{[n]\setminus [4]}{2}$, following \cite[Sec.\ 2.7.2]{landsberg2011tensors}, we identify $\psi_{i_1i_2,j_1j_2}$ with the $2\times 2$ minor of the generic matrix $Z=(z_{ij}) \in U\otimes W$ determined by the rows $i_1,i_2\in [4]$ and columns $j_1,j_2\in [n]\setminus [4]$:
\begin{align}\label{eq:t's as minors}
\psi_{i_1i_2,j_1j_2} & = z_{i_1j_1}z_{i_2j_2} - z_{i_1j_2}z_{i_2j_1} \nonumber \\
& = (e_{i_1}\otimes f_{j_1})(e_{i_2}\otimes f_{j_2}) - (e_{i_1}\otimes f_{j_2})(e_{i_2}\otimes f_{j_1}) \nonumber \\
& = 2(e_{i_1}\wedge e_{i_2})\otimes (f_{j_1}\wedge f_{j_2})  \in \textstyle{\bigwedge^2}U \otimes \textstyle{\bigwedge^2}W,
\end{align}
where the second equality is obtained by looking at the entry $z_{pq}$ has the coefficient of the element $e_p\otimes f_q \in U\otimes W$. We denote the skewsymmetric matrices $\underline{e} \in \bigwedge^2U$ and $\underline{f}\in \bigwedge^2W$. This is a combinatorial identification that, in geometric perspective, is equivalent to impose that the variables $\psi_{ijpq}$'s are in the image of the rational map 
\[ \mathbb P(U \otimes W)\dashrightarrow \mathbb P\left(\textstyle{\bigwedge^2U \otimes \bigwedge^2W}\right)\]
mapping a $4\times(n-4)$ matrix to the vector of its $2\times 2$ minors. 

\begin{remark}
This description does not apply in general. In fact, it would imply extra equations on the variety $V_P=V(P_J \mid J \in \binom{[n]\setminus [4]}{4}) \subset \mathbb P(\bigwedge^2U \otimes \bigwedge^2W)$, resulting actually in a subvariety of the truncation variety that we are interested in. The Zariski closure of the image of the above rational map is an example of {\em compound/Kalman variety} \cite{kalman2025} or {\em $2$nd collineation variety of tensors} \cite{gesmundo2025collineation}.
\end{remark}

\begin{remark}
Using \eqref{eq:FCI} and \eqref{eq:cpsi}, we can look at the identification of the variables $\psi_{i_1i_2j_1j_2}$ as $2\times 2$ minors in terms of creation-annihilation transformations. Indeed, using formulas \eqref{eq: def of X_ij} and \eqref{eq: def of X_IJ} one can show that, for any $i< j \in [4]$ and $p< q \in [n]\setminus [4]$, it holds
\[ 
X_{ip}X_{jq}e_{[4]} = - X_{iq}X_{jp}e_{[4]} \implies \big(X_{ip}X_{jq} - X_{iq}X_{jp}\big) e_{[4]} = 2 X_{ij,pq}e_{[4]} .
\]
Such relations hold also when applying the above excitation matrices to another state $e_I$ for $I \in \binom{[n]}{4}$ such that $I \neq [4]$. Indeed, if $\{i,j\} \not \subset I$ then both sides of the equation vanish. On the other hand, if $I=\{i < j < a < b\}$ then both sides give
\begin{align*}
(X_{ip}X_{jq} - X_{iq}X_{jp})e_I & = X_{ip}(e_{q}\wedge (-1)e_{iab}) - X_{iq}(e_p \wedge (-1)e_{iab}) \\
& = X_{ip}(e_i\wedge e_q \wedge e_a \wedge e_b) - X_{iq}(e_i \wedge e_p \wedge e_a \wedge e_b) \\
& = 2 (e_p \wedge e_q \wedge e_a \wedge e_b) \\
2X_{ij,pq}e_I & = 2 X_{i,p}(e_i \wedge e_q \wedge e_a \wedge e_b) = 2 (e_p \wedge e_q \wedge e_a \wedge e_b) .
\end{align*}
\end{remark}

Now, fix $J=\{j_1,j_2,j_3,j_4\}\in \binom{[n]\setminus [4]}{4}$ and denote by $[\underline{f}]_{J}$ the $4\times 4$ submatrix of $\underline{f}$ determined by rows and columns indexed by $J$. In the following computation we use that the symmetric product $T\cdot S$ of two tensors is by definition $\frac{1}{2}(T\otimes S + S\otimes T)$ (see $\clubsuit$), and that $E \otimes F \otimes E \otimes F\simeq E \otimes E \otimes F \otimes F$ for any $E,F$ vector spaces (see $\spadesuit$). Then starting from \eqref{eq:t's as minors}, one gets

{\scriptsize
\begin{align}\label{eq:binomials in t's}
    & \psi_{i_1i_2j_1j_2}\psi_{i_3i_4j_3j_4} + \psi_{i_1i_2j_3j_4}\psi_{i_3i_4j_1j_2} \nonumber \\
    & \! \! \! \stackrel{\eqref{eq:t's as minors}}{=} 4\bigg( \big[ (e_{i_1}\wedge e_{i_2})\otimes (f_{j_1}\wedge f_{j_2}) \big] \cdot \big[ (e_{i_3}\wedge e_{i_4})\otimes (f_{j_3}\wedge f_{j_4}) \big] + \big[ (e_{i_1}\wedge e_{i_2})\otimes (f_{j_3}\wedge f_{j_4}) \big] \cdot \big[ (e_{i_3}\wedge e_{i_4})\otimes (f_{j_1}\wedge f_{j_2}) \big] \bigg) \nonumber \\ 
    & \stackrel{\clubsuit}{=} 4\bigg( \ 
        \frac{1}{2} \bigg[ (e_{i_1}\wedge e_{i_2})\otimes (f_{j_1}\wedge f_{j_2}) \otimes (e_{i_3}\wedge e_{i_4})\otimes (f_{j_3}\wedge f_{j_4}) 
        +  (e_{i_3}\wedge e_{i_4})\otimes (f_{j_3}\wedge f_{j_4}) \otimes (e_{i_1}\wedge e_{i_2})\otimes (f_{j_1}\wedge f_{j_2}) \bigg] \  \nonumber\\
        & \quad + \ \frac{1}{2}\bigg[ (e_{i_1}\wedge e_{i_2})\otimes (f_{j_3}\wedge f_{j_4}) \otimes (e_{i_3}\wedge e_{i_4})\otimes (f_{j_1}\wedge f_{j_2}) 
        +  (e_{i_3}\wedge e_{i_4})\otimes (f_{j_1}\wedge f_{j_2}) \otimes (e_{i_1}\wedge e_{i_2})\otimes (f_{j_3}\wedge f_{j_4}) \bigg] \ \bigg) \nonumber\\
    & \stackrel{\spadesuit}{=}  2 \bigg( \ 
        \bigg[ (e_{i_1}\wedge e_{i_2}) \otimes (e_{i_3}\wedge e_{i_4})\otimes (f_{j_1}\wedge f_{j_2})\otimes (f_{j_3}\wedge f_{j_4}) 
        +  (e_{i_3}\wedge e_{i_4}) \otimes (e_{i_1}\wedge e_{i_2})\otimes (f_{j_3}\wedge f_{j_4})\otimes (f_{j_1}\wedge f_{j_2}) \bigg] \  \nonumber\\
        & \quad + \ \bigg[ (e_{i_1}\wedge e_{i_2}) \otimes (e_{i_3}\wedge e_{i_4})\otimes (f_{j_3}\wedge f_{j_4})\otimes (f_{j_1}\wedge f_{j_2}) 
        +  (e_{i_3}\wedge e_{i_4}) \otimes (e_{i_1}\wedge e_{i_2})\otimes (f_{j_1}\wedge f_{j_2})\otimes (f_{j_3}\wedge f_{j_4}) \bigg] \ \bigg) \nonumber\\
    & =  2\bigg( \ 
        \bigg[ (e_{i_1}\wedge e_{i_2}) \otimes (e_{i_3}\wedge e_{i_4}) +  (e_{i_3}\wedge e_{i_4}) \otimes (e_{i_1}\wedge e_{i_2}) \bigg] \otimes \bigg[ (f_{j_1}\wedge f_{j_2})\otimes (f_{j_3}\wedge f_{j_4}) + (f_{j_3}\wedge f_{j_4})\otimes (f_{j_1}\wedge f_{j_2}) \bigg] \ \bigg) \nonumber\\
    & \stackrel{\clubsuit}{=} 2\bigg( \ \bigg[ 2 (e_{i_1}\wedge e_{i_2}) \cdot (e_{i_3}\wedge e_{i_4}) \bigg] \otimes \bigg[ 2 (f_{j_1}\wedge f_{j_2})\cdot (f_{j_3}\wedge f_{j_4}) \bigg] \ \bigg) \nonumber\\
    & = 8 \big[ \ (e_{i_1}\wedge e_{i_2}) \cdot (e_{i_3}\wedge e_{i_4}) \otimes (f_{j_1}\wedge f_{j_2})\cdot (f_{j_3}\wedge f_{j_4}) \ \big] .
    \end{align}
    }

\noindent Combining the expression of a $4\times 4$ Pfaffian in \eqref{eq:Pfaffian as tensor} and the equality \eqref{eq:binomials in t's} leads to
{\footnotesize
\begin{align*}
& 8\Pf(\underline{e})\otimes \Pf([\underline{f}]_J) \nonumber \\
& = 8 \bigg( (e_{1}\wedge e_2)(e_3\wedge e_4) - (e_1\wedge e_3)(e_2\wedge e_4) + (e_1\wedge e_4)(e_2\wedge e_3) \bigg) \\
& \quad \otimes 
\bigg( (f_{j_1}\wedge f_{j_2})(f_{j_3}\wedge f_{j_4}) - (f_{j_1}\wedge f_{j_3})(f_{j_2}\wedge f_{j_4}) + (f_{j_1}\wedge f_{j_4})(f_{j_2}\wedge f_{j_3}) \bigg) \nonumber \\
& = \psi_{12j_1j_2}\psi_{34j_3j_4} + \psi_{12j_3j_4}\psi_{34j_1j_2} - \psi_{12j_1j_3}\psi_{34j_2j_4} - \psi_{12j_2j_4}\psi_{34j_1j_3} + \psi_{12j_1j_4}\psi_{34j_2j_3} + \psi_{12j_2j_3}\psi_{34j_1j_4} \nonumber \\
& \quad  - \psi_{13j_1j_2}\psi_{24j_3j_4} - \psi_{13j_3j_4}\psi_{24j_1j_2} + \psi_{13j_1j_3}\psi_{24j_2j_4} + \psi_{13j_2j_4}\psi_{24j_1j_3} - \psi_{13j_1j_4}\psi_{24j_2j_3} - \psi_{13j_2j_3}\psi_{24j_1j_4} \nonumber \\
& \quad + \psi_{14j_1j_2}\psi_{23j_3j_4} + \psi_{14j_3j_4}\psi_{23j_1j_2} - \psi_{14j_1j_3}\psi_{23j_2j_4} - \psi_{14j_2j_4}\psi_{23j_1j_3} + \psi_{14j_1j_4}\psi_{23j_2j_3} + \psi_{14j_2j_3}\psi_{23j_1j_4} .\nonumber 
\end{align*}
}
\noindent Comparing with the definition of $P_J$ in \eqref{eq:PJ} we obtain
\begin{equation*}\label{eq:PJ as Pfaffian}
P_J = 8\Pf(\underline{e})\otimes \Pf([\underline{f}]_J),
\end{equation*}
showing that for disconnected doubles the quadrics defining the CCD truncation variety are indeed pure tensor products of Pfaffians.

\section{The Coupled Cluster Doubles Degree}\label{sec:ccd}

The truncation variety serves as a mean to simplify the high-dimensional Schrödinger equation \eqref{eq:schroedinger} (see Section~\ref{sec:background}). The \emph{coupled cluster (CC) degree} of the truncation variety $V_{\{2\}}$ is defined as the number of solutions to the (truncated) nonlinear eigenvalue problem
\begin{equation}
\label{eq:CCeq}
    (H\psi)_2 = \lambda\psi_2, \quad \psi \in V_{\{2\}},
\end{equation}
known as the unlinked CC equations. To access the CC degree, the Hamiltonian $H$ is assumed to be a generic $\binom{n}{4} \times \binom{n}{4}$ matrix indexed by the subsets of $n$ of size $4$ and $(\cdot)_2$ denotes the projection to coordinates indexed by subsets $I$ of level $|I \backslash [4]| = 2$. 

By \cite[Theorem~5.2]{faulstich2024algebraic} we have an upper bound for the CC degree of a variety, using its dimension and degree. We know that $\dim( V_{\{2\}})= 6 \cdot \binom{n - 4}{2}$. If, as conjectured,  $\deg V_{\{2\}} = 2^{n-4 \choose 4}$ for $n \leq 12$, then, in these cases,

\begin{equation}\label{eq:bound}
    \operatorname{CCdeg}_{4,n}(\{2\}) \le (\dim( V_{\{2\}}) + 1)\deg(V_{\{2\}}) = \left( 6 \cdot \binom{n - 4}{2} +1 \right) 2^{n-4 \choose 4}.
\end{equation}

\begin{example}
\label{ex: CC bound}
CC degrees of $V_{\{2\}}$ for $n = 8,9,10$ and their corresponding upper bounds.

\begin{center}
    \begin{tabular}{c|l|l|l|l}
    \toprule
         n & $\dim V_{\{2\}}$ & $\deg V_{\{2\}}$ & bound & $\CCdeg_{4,n}(\{2\})$  \\
    \midrule
         8 & 36 & $2^1=2$ & 74 & 73 \\
        9 & 60 & $2^5=36$ & 1{,}925 & 1{,}823 \\
        10 & 90 & $2^{15}=32768$ & 2{,}981{,}888 & 2{,}523{,}135 \\
    \bottomrule
    \end{tabular}
\end{center}

\end{example}

\noindent While we are able to numerically compute the CC degree for small $n$ (see Example~\ref{ex: CC bound}), such explicit computations rapidly become numerically intractable. Hence, we require a numerically more tractable bound for practical purposes.
By \cite[Theorem~26]{sverrisdottir2025algebraic}, we can also define the CC degree of $V_{\{2\}}$ to be the total degree of the graph of the restricted map 
\[ \begin{matrix}
    \mathcal V_{\{2\}} & \longrightarrow & \mathcal H \\
    x & \mapsto &  \psi = \exp T(x) e_{[4]} 
    \end{matrix} \quad \textrm{with} \quad \psi_I
= \begin{cases}
    1 & \text{if $I$ is of level $0$ (i.e., $I=[4]$),}\\
    x_{I} & \text{if $I$ is of level $2$,}\\
    P_I(x) & \text{if $I$ is of level $4$ (i.e., $I \in \binom{[n]\setminus [4]}{4}$)}, \\
    0 & \text{otherwise.}
\end{cases} 
\]
Let $m= 6 \cdot {{n-4} \choose 2}$ be the number of level 2 coordinates and $s={{n-4} \choose 4}$ be the number of level 4 coordinates.
Since some $\psi_I$ are zero (at level $1$ and $3$), it suffices to look at the corresponding projective map with restricted image:
$$ \varphi: \mathbb P^m  \dashrightarrow \mathbb P^{m+s} $$
such that
\[ \begin{matrix}
\varphi: & \mathbb P^m \setminus \textrm{V}(x_{[4]}) & \longrightarrow & \mathbb P^{m+s} \\
& x & \mapsto & \psi 
\end{matrix} \quad \text{with} \quad \psi_I = \begin{cases}
    x_{[4]}^2 & \text{if $I$ is of level $0$,}\\
    x_{I}x_{[4]} & \text{if $I$ is of level $2$,}\\
    P_I(x) & \text{if $I$ is of level $4$}.
\end{cases} 
\]
The graph of $\varphi$ is $G:= \{(x,\varphi(x)) \mid x \in \mathbb P^m \} \subset \mathbb P^m \times \mathbb P^{m+s}$ and its dimension is $m$. Its image is $\text{im}(\varphi) = V_{\{2\}}$, which has dimension $m$ as well. In order to obtain a better understanding of what the degree of this graph might be, we start out with degrees of generic graphs of this form. Based on computations carried out with the code in \cite{Faulstich2025SupplementaryCode}, we conjecture the following:

\begin{conjecture}
Consider the rational map $\gamma: \mathbb P^m \dashrightarrow \mathbb P^{m+s}$ such that 
\[ \begin{matrix}
\gamma: & \mathbb P^m \setminus \textrm{V}(z_0) & \longrightarrow & \mathbb P^{m+s}\\
& z & \mapsto & y 
\end{matrix} \quad \text{with} \quad y_i = \begin{cases}
    z_0^2 & \text{if }i=0,\\
    z_0z_i & \text{if } i \in \{1,\ldots,m\},\\
    \gamma_i(z) & \text{if } i \in \{ m+1, \ldots, m+s\},
\end{cases}
\]
where $\gamma_i(z)$ are general homogeneous quadratic polynomials for $i=m+1,\ldots, m+s$. Then the graph of $\gamma$ has degree
\[ \deg(\text{graph}(\gamma)) = \begin{cases}
(m - s + 2)2^s -1 & \text{if } s<m, \\
2^{m+1}-1 & \text{if } s \geq m.
\end{cases} \]
\end{conjecture}

\noindent If $V_P$ is a complete intersection for $n \leq 12$, this would suggest the following:

\begin{conjecture}\label{conj: CC degree}
    For $d=4$ and $n \leq 12$, the CC degree of the truncation variety $V_{\{2\}}$ is 
    $$ \CCdeg_{4,n}(\{2\}) =
    (m - s + 2)2^s -1,$$ where $m = 6{{n-4}\choose 2}$ and $s={{n-4}\choose 4}$.
\end{conjecture}

\noindent Indeed, this is confirmed in Example~\ref{ex: CC bound} for small instances of $n \leq 10$. If the conjecture is true, this implies that for $n=11$ the CC degree is $ 3{,}195{,}455{,}668{,}223$ and for $n=12$ it is $1.18 \cdot 10^{23}$.
\begin{center}
    \begin{tabular}{c|l|l|l|l|l}
         n & m & s & $(m-s+2)2^s-1$ & $\CCdeg$ & bound \\
         \hline
         8 & 36 & 1 & 73 & 73 & 74\\
        9 & 60 & 5 & 1{,}823 & 1{,}823 & 1{,}935\\
        10 & 90 & 15 & 2{,}523{,}135 & 2{,}523{,}135 & 2{,}981{,}888 \\
        11 & 126 & 35 & 3{,}195{,}455{,}668{,}223 &? & 4{,}363{,}686{,}772{,}736\\
        12 & 168 & 70 & 118{,}059{,}162{,}071{,}741{,}130{,}342{,}399 & ?& 199{,}519{,}983{,}901{,}242{,}510{,}278{,}656
    \end{tabular}
\end{center}

\section{Quantum Chemistry Simulation}\label{Section 5}

The algebraic properties conjectured above for the truncation variety apply to four electrons in up to twelve spin-orbitals, that is, a $\mathrm{CAS}(4,6)$ in quantum chemistry parlance. This is the smallest electronic regime in which the CCD truncation becomes genuinely nonlinear and already exhibits nontrivial correlation effects, while still being small enough that full solution sets can be enumerated and validated. The purpose of this section is to connect these structural statements to concrete simulation scenarios and to explain how they inform practice, for instance, by anticipating the multiplicity and qualitative nature of CCD solutions and by indicating when iterative solvers may be affected by multiple roots or singular behavior. The simulations presented here were performed using the software packages \texttt{HomotopyContinuation.jl}~\cite{breiding2018homotopycontinuation} and \texttt{PySCF}~\cite{sun2015libcint,sun2018pyscf,sun2020recent}.

While lithium hydride and the H$_4$ model are common four-electron systems to investigate the full CC solution set using algebraic geometry techniques, here we turn to a chemically more challenging case to interrogate the underlying algebraic structure: the beryllium insertion into molecular hydrogen. As the beryllium atom (Be) approaches the hydrogen molecule (H$_2$), an H--Be--H bond forms, and the electronic structure evolves from an essentially weak interaction between subsystems to a strongly coupled bonding regime. It is well-documented that this crossover amplifies multireference effects and can severely stress single-reference correlation methods~\cite{purvis1983c2v}. At the same time, the system remains small enough that the full CCD solution set can be examined in detail, making it a useful test case relating the geometry of the truncation variety $V_{\{2\}}$ to observed computational behavior of computationally challenging systems.

\begin{figure}[htbp]
    \centering
    \includegraphics[width = 0.6\textwidth]{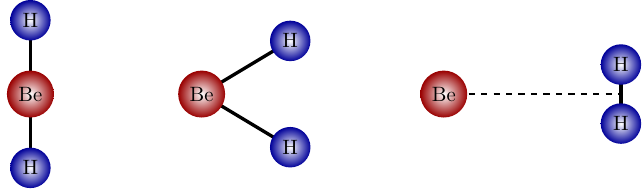}
    \caption{Depiction of the C$_{2v}$ insertion pathway for BeH$_2$. From left to right: X-position equal to 0, 3.7, and 5.0.}
    \label{fig:BeH2_mod}
\end{figure}

Specifically, we consider the C$_{2v}$ insertion pathway for BeH$_2$ (Be $\cdots$ H$_2$). We denote the position of
the beryllium atom by X-position, where X-position equal to zero corresponds to the linear equilibrium state and X-position equal to five corresponds to the non-interacting subsystems, see Figure~\ref{fig:BeH2_mod}. Along this pathway, the transition region features a switch in the dominant determinant: two leading determinants govern different segments of the potential energy surface and become nearly degenerate near the transition geometry. This near degeneracy leads to apparent discontinuities if one insists on a fixed reference description. Smoother potential energy surfaces can be obtained by explicitly switching the reference across the critical region~\cite{evangelista2011alternative,bodenstein2020state}. In this work, however, we are interested in the full CC solution set across this transition. By tracking all solutions through the near-degeneracy where the dominant determinant changes, we expose the associated algebraic structure beyond a single adiabatically followed CC root.

A complementary qualitative perspective is provided by the frontier molecular orbitals in Figure~\ref{fig:MOsBeH2}. In the insertion and dissociation limits, the highest occupied molecular orbital (HOMO) and lowest unoccupied molecular orbital (LUMO) retain comparatively clean and weakly mixed character, consistent with a dominant reference configuration. Near $X=3.7$, by contrast, the frontier orbitals delocalize across the forming bond and change character, reflecting a reduced HOMO--LUMO gap and enhanced configuration mixing. This is precisely the regime in which the CCD residual map becomes more nonlinear and more ill-conditioned, so that multiple nearby roots and solver pathologies such as slow convergence, root switching, or divergence become likely.

\begin{figure}[htbp]
    \centering
    \includegraphics[width=0.7\linewidth]{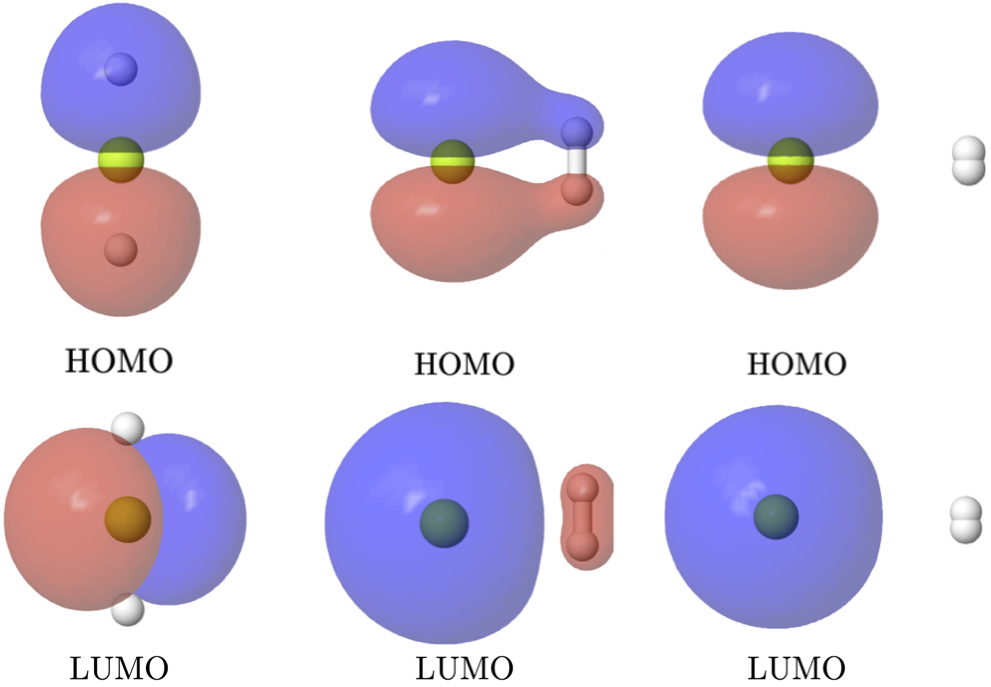}
    \caption{Molecular orbitals along the Be insertion coordinate in BeH$_2$. The top row shows the highest occupied molecular orbital (HOMO), the bottom row shows the lowest unoccupied molecular orbital (LUMO). From left to right: fully inserted geometry at $X=0$, intermediate geometry at $X=3.7$, and separated geometry at $X=5.0$. }
    \label{fig:MOsBeH2}
\end{figure}

Despite the intricate valence electronic structure, the beryllium 1s electrons are essentially inert along the pathway. This is reflected in the molecular orbital spectrum, which shows a pronounced energy gap separating the core orbital from the remaining (valence) orbitals, see Figure~\ref{fig:BeH2Energy}. This naturally motivates a $\sigma$-only active space, with the beryllium 1s electrons treated as frozen core and excluded from the correlated treatment. This reduction yields an effective four-electron problem, consistent with the scope of the theory developed in this article.

\begin{figure}[htbp]
    \centering
    \includegraphics[width=0.7\linewidth]{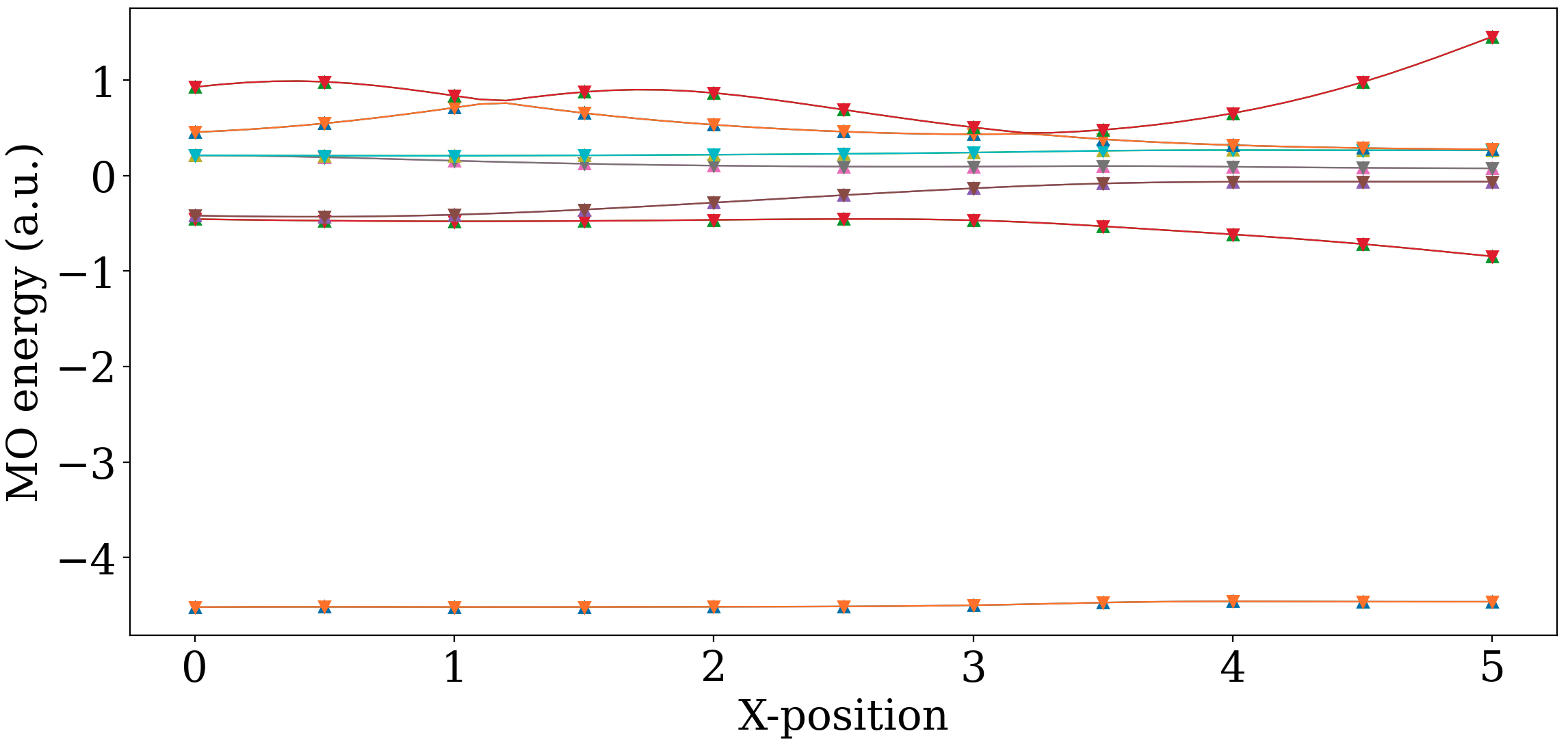}
    \caption{Molecular orbital energies along the beryllium insertion pathway into molecular hydrogen.}
    \label{fig:BeH2Energy}
\end{figure}

We begin our numerical investigations by correlating the four electrons in eight spin orbitals selected as the energetically lowest-lying molecular orbitals in Figure~\ref{fig:BeH2Energy} (excluding the Be 1s). The corresponding Galerkin space is therefore of dimension ${8 \choose 4} = 70$. Diagonalization of the Hamiltonian yields the potential energy curves shown in Figure~\ref{Fig:4Elein8SpO}a, while computing the full CC solution set along the pathway produces the solutions shown in Figure~\ref{Fig:4Elein8SpO}b. We note that the potential energy curves exhibit discontinuities at the geometry where the reference determinants cross. This feature arises because our analysis is performed in a molecular-orbital representation, obtained by transforming the Hamiltonian at each geometry to resolve the local CC-root structure in the vicinity of the crossing.

\begin{figure}[htbp]
    \centering
    \includegraphics[width = 0.46\textwidth]{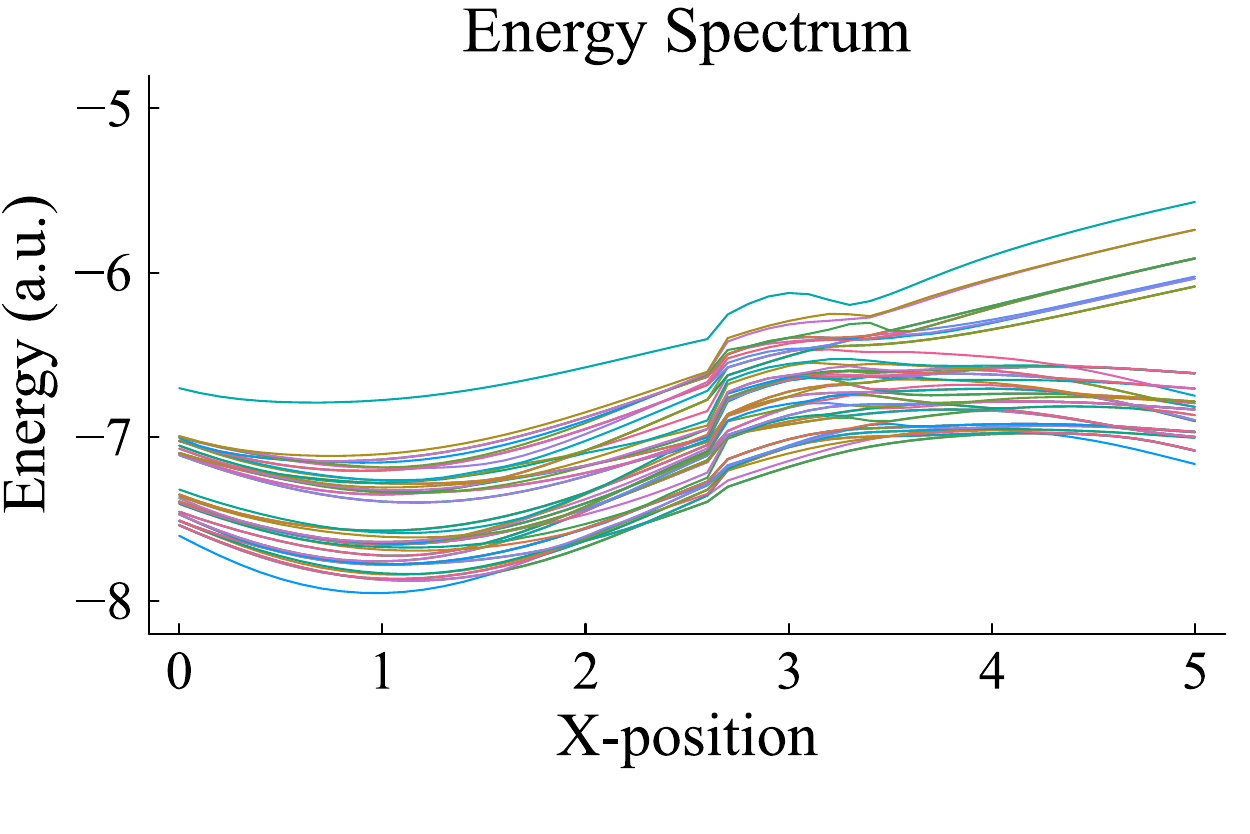}  
    \includegraphics[width = 0.46\textwidth]{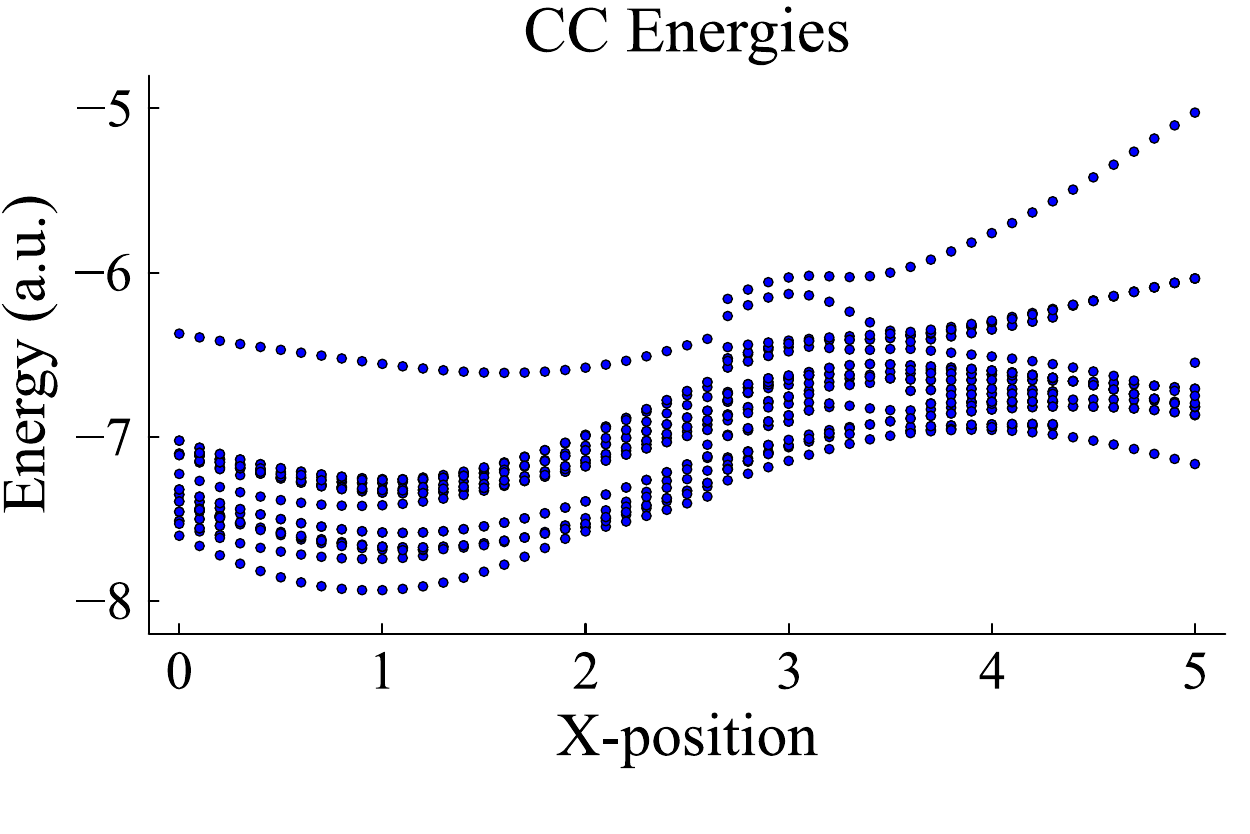}  
    \caption{(a) Full energy spectrum of the Hamiltonian in molecular orbital basis (b) Real-valued CCD energy solutions}
    \label{Fig:4Elein8SpO}
\end{figure}

For a generic Hamiltonian in the four-electron, eight–spin-orbital setting, the associated CC polynomial system admits 73 solutions, see Example~\ref{ex: CC bound}. To track the root count along the reaction pathway, we employ a homotopy continuation strategy that deforms the generic system into the corresponding chemical Hamiltonian. Along this deformation, the total number of solutions typically drops by roughly a factor of 1/2, suggesting that the generic Hamiltonian family effectively overparameterizes the solution structure of the chemical system. Although we solve the polynomial system over $\mathbb{C}$, it is notable that, despite the appearance of complex roots, the corresponding energies are predominantly real, see Figure~\ref{fig:RootCount4in8}. Notably, an increase in the latter two coincides with the state crossings visible in Figure~\ref{Fig:4Elein8SpO}a. This alignment indicates that, in chemically challenging regions, the truncation variety itself becomes more intricate, leading to an expansion and reorganization of the coupled-cluster solution manifold and a corresponding proliferation of roots.  

\begin{figure}[htbp]
    \centering
    \includegraphics[width = 0.45\textwidth]{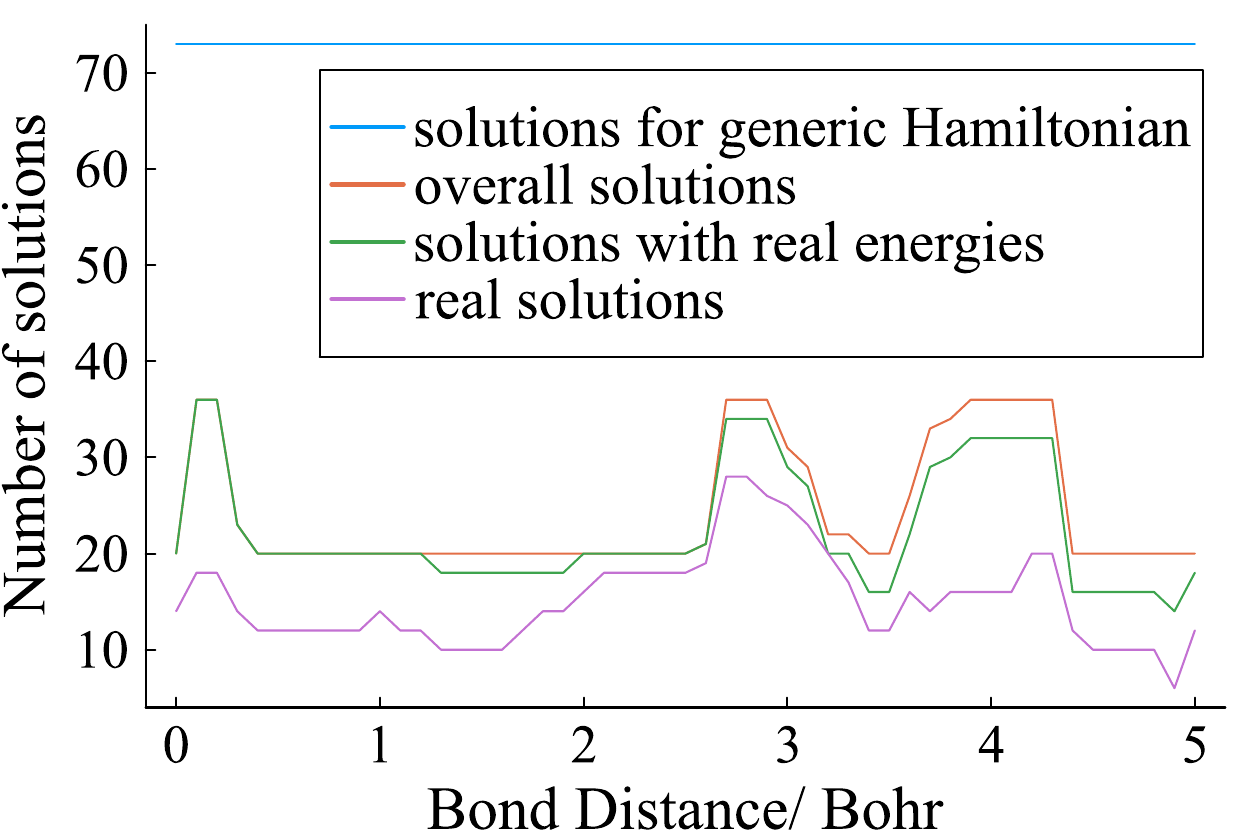}
    \caption{Different CC solution counts along the beryllium insertion pathway into molecular hydrogen.}
    \label{fig:RootCount4in8}
\end{figure}

While the four-electron, eight–spin-orbital case provides a minimal testbed, it is too small to fully resolve the geometry of the truncation variety and the rapid growth in the number of solutions. To probe these effects more clearly, we enlarge the model by including two additional spin orbitals. In this enlarged setting, the corresponding generic system admits 2,523,135 solutions, and the homotopy continuation procedure required to compute the full solution set for a single Hamiltonian takes approximately a week with \verb|Julia| on the server\footnote{2× Intel Xeon Gold 6240 CPUs with 36 cores/72 threads per CPU, 1 TB DDR4 RAM, and NVMe SSD storage.} on 144 threads.
Consequently, we restrict attention to $X=2.5,2.6$, and $2.7$ Bohr to examine root distributions in the vicinity of the first reference crossing. The corresponding solution counts are summarized in Table ~\ref{tab:SolutionCount4in10}.

\begin{table}[htbp]
    \centering
    \begin{tabular}{r|ccc}
    \toprule
        $X$-position   &  2.5 & 2.6 & 2.7 \\
        \midrule
        $N_{\mathrm{gen}}$  &  2,523,135 & 2,523,135 & 2,523,135\\
        \midrule
        $N_{\mathrm{chem}}$ & 16268 &16876 & 17879\\
        $N_{\mathrm{chem}}^{E\in\mathbb{R}}$ & 3637 & 3844 & 3352\\
        $N_{\mathrm{chem}}^{\mathrm{real}}$ & 958 & 965 & 968\\
        \bottomrule
    \end{tabular}
    \caption{CC solution counts: $N_{\mathrm{gen}}$  root count of the generic Hamiltonian; $N_{\mathrm{chem}}$ root count of the corresponding chemical Hamiltonian; $N_{\mathrm{chem}}^{E\in\mathbb{R}}$ the number of roots whose associated energies are real; and $N_{\mathrm{chem}}^{\mathrm{real}}$ number of real-valued roots.}
    \label{tab:SolutionCount4in10}
\end{table}

We note the overall increase in the number of solutions describing the chemical system, which is consistent with the observation in Figure~\ref{fig:RootCount4in8}. Interestingly, the number of solutions that yield valued-energies (which is a physical requirement) strongly decreases when transitioning from $X = 2.6$ to $X=2.7$, where it reaches its minimum among the three geometries considered. At $X=2.7$, the ground and first excited state (see Figure~\ref{Fig:4Elein8SpO}a) cross and so do the reference determinants. In earlier work~\cite{faulstich2024coupled}, we observed that as the reference state deteriorates, measured by its overlap with the targeted eigenstate, the number of solutions with complex energies increases. While the former study was based on a simplified model, we observe the same behavior here, confirming the behavior for physical systems. A complementary and somewhat surprising feature is that the number of solutions with real-valued amplitudes, $N_{\mathrm{chem}}^{\mathrm{real}}$, remains nearly constant across all three geometries. Since real amplitudes necessarily yield real energies, this indicates that the loss of real-energy solutions near the crossing is not driven by a disappearance of real roots in amplitude space. Instead, it is the subset of complex-amplitude solutions that yield real energies that shrinks as $R$ increases from 2.6 to 2.7. Hence, the reference crossing destabilizes the complex solutions that are physically admissible at smaller bond lengths. This suggests that while the real-amplitude solution set appears structurally robust, the ability of complex solutions to correspond to physical states is highly sensitive to changes in the reference.

Although the importance of reference quality is well appreciated in computational chemistry, its effect on the global coupled-cluster solution set is poorly understood. In particular, it is not clear why degradation of the reference, for instance near determinant crossings, is accompanied by a substantial change in the root distribution and, specifically, by a reduced fraction of solutions that yield physically admissible real energies.

Lastly, we plot the associated coupled-cluster energies in Figure~\ref{fig:CCDEnergies_4in8}. The figure displays both real- and complex-valued energies of the chemical Hamiltonian in the complex plane. In agreement with previous studies~\cite{faulstich2024algebraic,sverrisdottir2024exploring}, the spectra exhibit pronounced clustering in distinct regions of the energy plane rather than a uniform distribution. The quantitative decrease in $N_{\mathrm{chem}}^{E\in\mathbb{R}}$ is also visible at a qualitative level. As the geometry approaches the reference crossing, the cloud of energy values broadens in the imaginary direction, reflecting the growing proportion of solutions with nonzero imaginary parts. This visual spreading of the spectrum mirrors the reduction in physically admissible, real-energy solutions reported in Table~\ref{tab:SolutionCount4in10}.

\begin{figure}[H]
    \includegraphics[width = 0.32\textwidth]{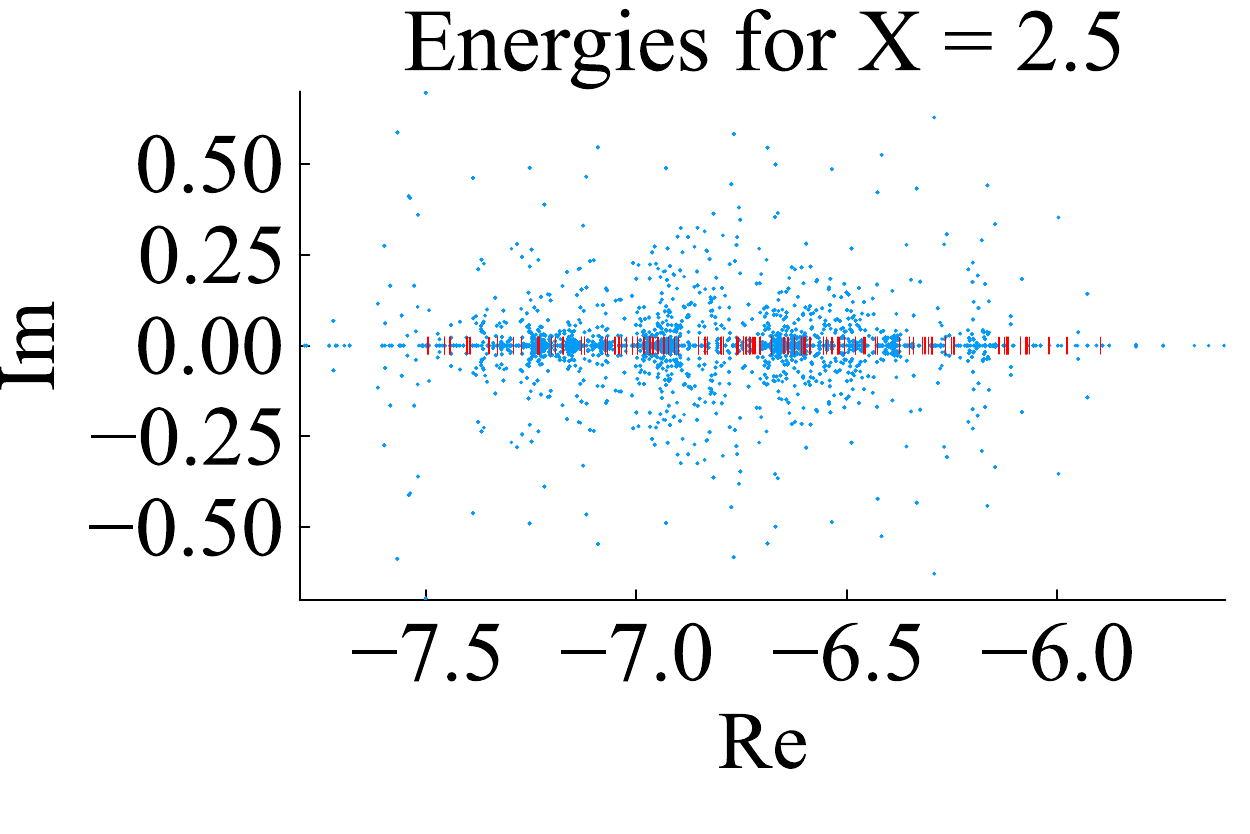} 
    \includegraphics[width = 0.32\textwidth]{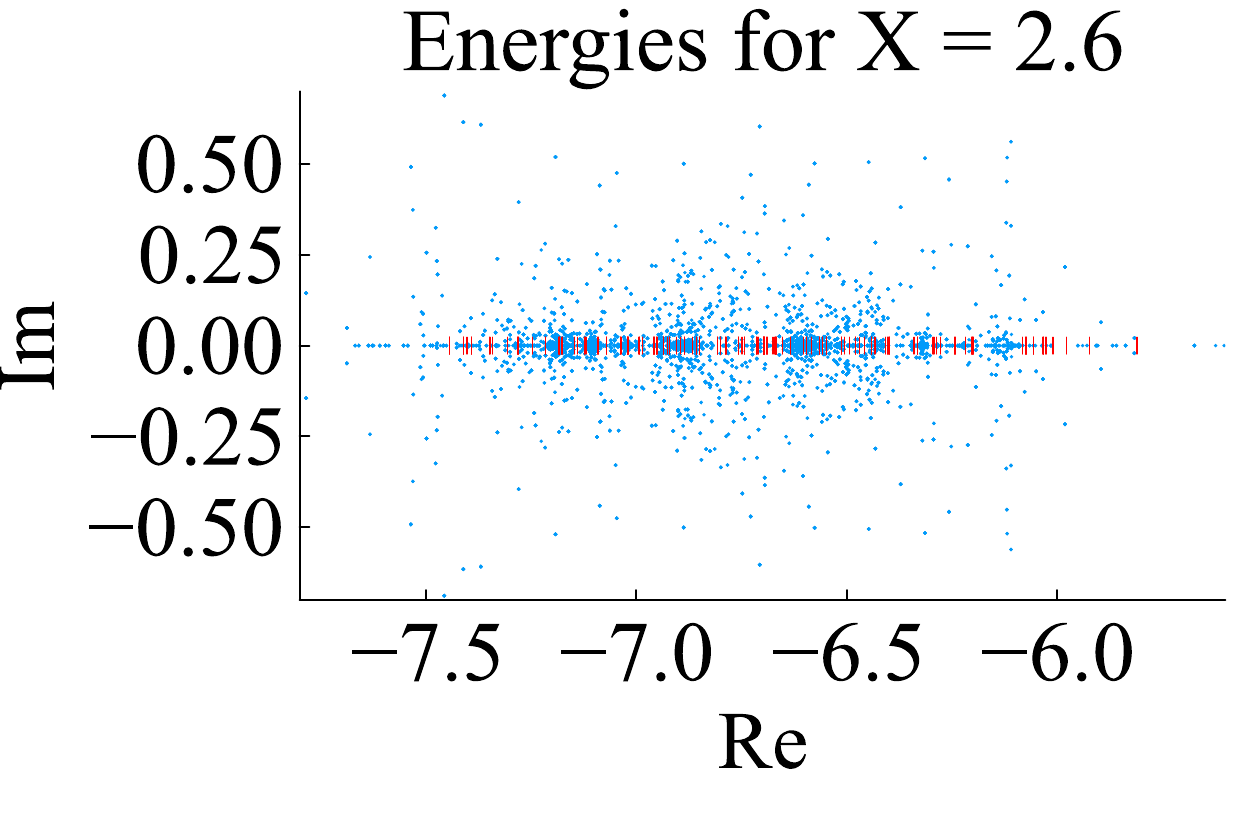}
    \includegraphics[width = 0.32\textwidth]{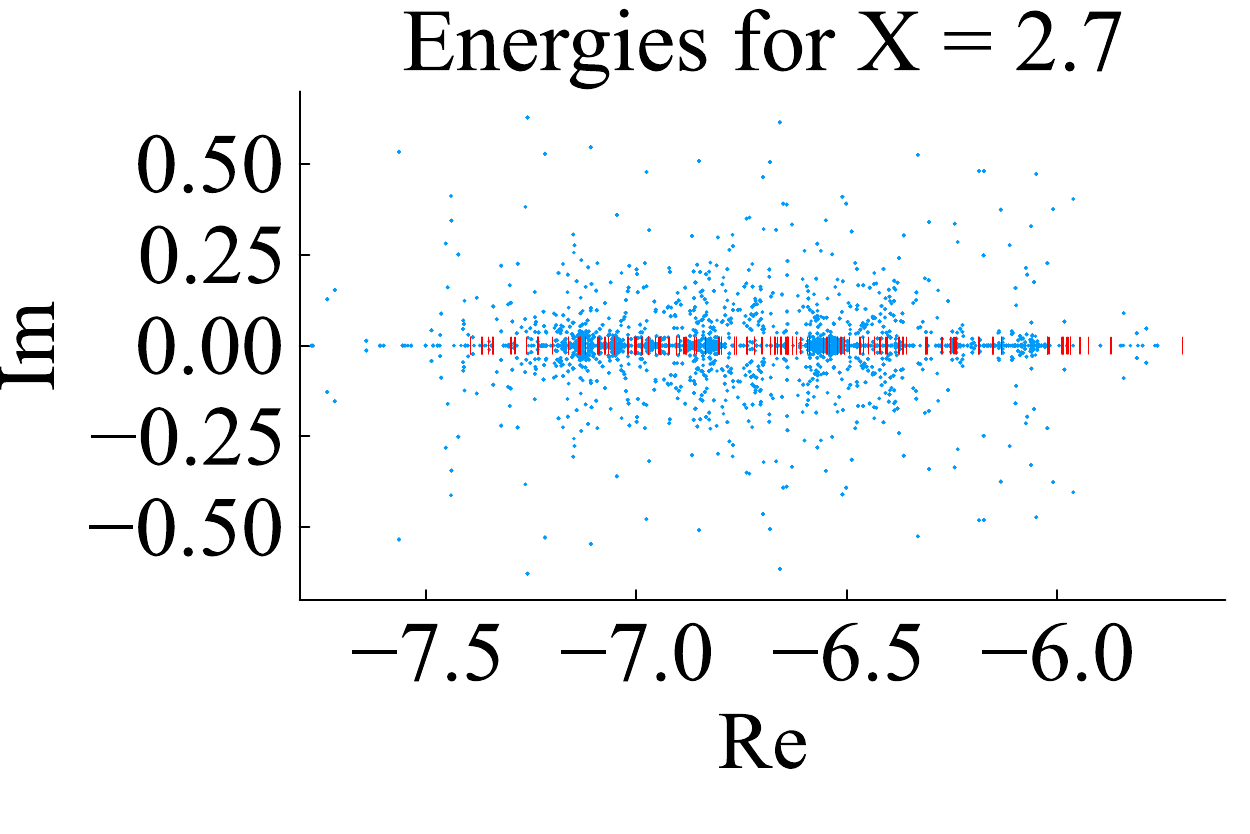}
    \caption{CC (blue dots) and exact (red lines) energies for different $X$-positions. From left to right: $X=2.5$, $X=2.6$ and $X=2.7$. }
    \label{fig:CCDEnergies_4in8}
\end{figure}

We conclude the numerical investigations by emphasizing a conceptual point that appears central for understanding the physically relevant root structure of coupled-cluster theory. From a computational algebraic viewpoint, one can classify coupled-cluster solutions by whether their amplitudes are real or complex. Physically, however, this is not the right dichotomy. What matters is whether a solution, with real or complex amplitudes, yields real observables, here the energy. The observed loss of physically admissible roots near the reference crossing is therefore not simply a change in the real locus of the solution set, but a change in the subset of roots on which the energy map attains real values.

This distinction also highlights an algebraic challenge. Classical algebraic geometry primarily studies the solution sets of polynomial systems over $\mathbb{C}$ or $\mathbb{R}$. Our setting suggests that, for applied algebraic and numerical purposes, an equally important object is the subset of a solution variety that is mapped into $\mathbb{R}$ by a distinguished objective function of interest. In our application, such functions describe physical observables, here exemplified by the energy. This subset can include solutions with genuinely complex amplitudes that nevertheless produce real valued outcomes under an objective function of interest.  

The need to understand such ``real-valued observable'' subsets has been noted repeatedly in applications. In the present chemical setting, this perspective is particularly relevant for explaining how geometric changes along a reaction coordinate, together with deteriorating reference quality, can sharply reduce the number of physically admissible solutions near state crossings. Addressing this phenomenon appears to require tools beyond root counting alone, it calls for a framework that analyzes how a chosen observable partitions the solution set and how this structure evolves as parameters change.

\subsubsection*{Acknowledgements}
We thank Bernd Sturmfels and Svala Sverrisd{\'o}ttir for their valuable inputs during the course of this paper and for their help with the codes, and Fulvio Gesmundo for helpful discussions on invariant theory. We also thank MPI-MiS Leipzig for hosting the workshop \textquotedblleft Nonlinear Algebra meets Quantum Chemistry\textquotedblright \space from which this collaboration started. VG is grateful to MPI-CBG Dresden and CSBD for the excellent working conditions. VG is member of the Italian national group GNSAGA-INdAM.

\newpage
\small
\bibliographystyle{plainnat}
\bibliography{lib.bib}

@article{price2026plane,
  title={Plane partitions and spin adapted quantum states},
  author={Price, Abigail and Stelzer, Ada and Sverrisd{\'o}ttir, Svala},
  journal={arXiv:2601.06295},
  year={2026}
}

@article{faulstich2026algebraic,
  title={Algebraic Geometry for Spin-Adapted Coupled Cluster Theory},
  author={Faulstich, Fabian M. and Sverrisd{\'o}ttir, Svala},
  journal={arXiv:2601.16646},
  year={2026}
}

@article{borovik2025numerical,
  title={Numerical algebraic geometry for energy computations on tensor train varieties},
  author={Borovik, Viktoriia and Friedman, Hannah and Ho{\c{s}}ten, Serkan and Pfeffer, Max},
  journal={arXiv:2512.06939},
  year={2025}
}

@article{harwood2022improving,
  title={Improving the variational quantum eigensolver using variational adiabatic quantum computing},
  author={Harwood, Stuart M. and Trenev, Dimitar and Stober, Spencer T. and Barkoutsos, Panagiotis and Gujarati, Tanvi P. and Mostame, Sarah and Greenberg, Donny},
  journal={ACM Transactions on Quantum Computing},
  volume={3},
  number={1},
  pages={1--20},
  year={2022},
  publisher={ACM New York, NY}
}

@article{pokhilko2025homotopy,
  title={Homotopy continuation method for solving {D}yson equation fully self-consistently: theory and application to {N}d{N}i{O}2},
  author={Pokhilko, Pavel and Zgid, Dominika},
  journal={arXiv:2507.00290},
  year={2025}
}

@article{cances2024mathematical,
  title={A mathematical analysis of {IPT-DMFT}},
  author={Canc{\`e}s, {\'E}ric and Kirsch, Alfred and Perrin-Roussel, Solal},
  journal={arXiv:2406.03384},
  year={2024}
}

@article{gontier2019numerical,
  title={Numerical construction of {W}annier functions through homotopy},
  author={Gontier, David and Levitt, Antoine and Siraj-Dine, Sami},
  journal={Journal of Mathematical Physics},
  volume={60},
  number={3},
  year={2019},
  publisher={AIP Publishing}
}

@article{sverrisdottir2024exploring,
  title={Exploring ground and excited states via single reference coupled-cluster theory and algebraic geometry},
  author={Sverrisd{\'o}ttir, Svala and Faulstich, Fabian M.},
  journal={Journal of Chemical Theory and Computation},
  volume={20},
  number={19},
  pages={8517--8528},
  year={2024},
  publisher={ACS Publications}
}

@article{evangelista2011alternative,
  title={Alternative single-reference coupled cluster approaches for multireference problems: The simpler, the better},
  author={Evangelista, Francesco A.},
  journal={J. Chem. Phys.},
  volume={134},
  number={22},
  pages={224102},
  year={2011},
  publisher={American Institute of Physics}
}

@article{bodenstein2020state,
  title={A state-specific multireference coupled-cluster method based on the bivariational principle},
  author={Bodenstein, Tilmann and Kvaal, Simen},
  journal={J. Chem. Phys.},
  volume={153},
  number={2},
  pages={024106},
  year={2020},
  publisher={AIP Publishing LLC}
}

@article{purvis1983c2v,
  title={{C2V} Insertion pathway for {BeH2}: A test problem for the coupled-cluster single and double excitation model},
  author={Purvis III, George D. and Shepard, Ron and Brown, Franklin B. and Bartlett, Rodney J.},
  journal={Int. J. Quantum Chem.},
  volume={23},
  number={3},
  pages={835--845},
  year={1983},
  publisher={Wiley Online Library}
}

@article{faulstich2023homotopy,
  title={Homotopy continuation methods for coupled-cluster theory in quantum chemistry},
  author={Fabian M. Faulstich and Andre Laestadius},
  journal={Molecular Physics},
  pages={e2258599},
  year={2023},
  publisher={Taylor \& Francis}
}

@article{sun2020recent,
  title={Recent developments in the {PySCF} program package},
  author={Sun, Qiming and Zhang, Xing and Banerjee, Samragni and Bao, Peng and Barbry, Marc and Blunt, Nick S. and Bogdanov, Nikolay A. and Booth, George H. and Chen, Jia and Cui, Zhi-Hao and others},
  journal={The Journal of chemical physics},
  volume={153},
  number={2},
  year={2020},
  publisher={AIP Publishing}
}

@article{sun2018pyscf,
  title={{PySCF}: the {P}ython-based simulations of chemistry framework},
  author={Sun, Qiming and Berkelbach, Timothy C and Blunt, Nick S and Booth, George H. and Guo, Sheng and Li, Zhendong and Liu, Junzi and McClain, James D. and Sayfutyarova, Elvira R. and Sharma, Sandeep and others},
  journal={Wiley Interdisciplinary Reviews: Computational Molecular Science},
  volume={8},
  number={1},
  pages={e1340},
  year={2018},
  publisher={Wiley Online Library}
}

@article{sun2015libcint,
  title={Libcint: An efficient general integral library for {G}aussian basis functions},
  author={Sun, Qiming},
  journal={Journal of computational chemistry},
  volume={36},
  number={22},
  pages={1664--1671},
  year={2015},
  publisher={Wiley Online Library}
}

@article{laestadius2018analysis,
  title={Analysis of the extended coupled-cluster method in quantum chemistry},
  author={Laestadius, Andre and Kvaal, Simen},
  journal={SIAM Journal on Numerical Analysis},
  volume={56},
  number={2},
  pages={660--683},
  year={2018},
  publisher={SIAM}
}

@article{hassan2023analysis,
  title={Analysis of the single reference coupled cluster method for electronic structure calculations: the full-coupled cluster equations},
  author={Hassan, Muhammad and Maday, Yvon and Wang, Yipeng},
  journal={Numerische Mathematik},
  volume={155},
  number={1-2},
  pages={121--173},
  year={2023},
  publisher={Springer}
}

@article{hassan2023analysis2,
  title={Analysis of the Single Reference Coupled Cluster Method for Electronic Structure Calculations: The Discrete Coupled Cluster Equations},
  author={Hassan, Muhammad and Maday, Yvon and Wang, Yipeng},
  journal={arXiv:2311.00637},
  year={2023}
}

@article{jankowski1999physical4,
  title={Physical and mathematical content of coupled-cluster equations. {IV.} {I}mpact of approximations to the cluster operator on the structure of solutions},
  author={Jankowski, Karol and Kowalski, Karol},
  journal={The Journal of Chemical Physics},
  volume={111},
  number={7},
  pages={2952--2959},
  year={1999},
  publisher={American Institute of Physics}
}

@article{jankowski1999physical3,
  title={Physical and mathematical content of coupled-cluster equations. {III.} {M}odel studies of dissociation processes for various reference states},
  author={Jankowski, Karol and Kowalski, Karol},
  journal={The Journal of Chemical Physics},
  volume={111},
  number={7},
  pages={2940--2951},
  year={1999},
  publisher={American Institute of Physics}
}

@article{jankowski1999physical1,
  title={Physical and mathematical content of coupled--cluster equations: Correspondence between coupled--cluster and configuration--interaction solutions},
  author={Jankowski, Karol and Kowalski, Karol},
  journal={The Journal of Chemical Physics},
  volume={110},
  number={8},
  pages={3714--3729},
  year={1999},
  publisher={American Institute of Physics}
}

@article{jankowski1999physical2,
  title={Physical and mathematical content of coupled-cluster equations. {II.} {O}n the origin of irregular solutions and their elimination via symmetry adaptation},
  author={Jankowski, Karol and Kowalski, Karol},
  journal={The Journal of Chemical Physics},
  volume={110},
  number={19},
  pages={9345--9352},
  year={1999},
  publisher={American Institute of Physics}
}

@article{kowalski1998full,
  title={Full solution to the coupled-cluster equations: the {H}4 model},
  author={Kowalski, Karol  and Jankowski, Karol},
  journal={Chemical Physics Letters},
  volume={290},
  number={1-3},
  pages={180--188},
  year={1998},
  publisher={Elsevier}
}

@article{paldus1993application,
  title={Application of {H}ilbert-space coupled-cluster theory to simple {(H$_2$)$_2$} model systems: Planar models},
  author={Paldus, Josef and Piecuch, Piotr and Pylypow, Leszek and Jeziorski, Bogumil},
  journal={Physical Review A},
  volume={47},
  number={4},
  pages={2738},
  year={1993},
  publisher={APS}
}

@article{piecuch1990coupled,
  title={Coupled-cluster approaches with an approximate account of triexcitations and the optimized-inner-projection technique. {II.} Coupled-cluster results for cyclic-polyene model systems},
  author={Piecuch, Piotr and Zarrabian, Sohrab and Paldus, Josef and {\v{C}}{\'\i}{\v{z}}ek, Ji{\v{r}}{\'\i}},
  journal={Physical Review B},
  volume={42},
  number={6},
  pages={3351},
  year={1990},
  publisher={APS}
}

@article{csirik2023disc,
	author = {Mih{\'a}ly A. Csirik and Andre Laestadius},
	title = {Coupled-Cluster theory revisited - Part {I}: {D}iscretization},
	journal = {ESAIM: Mathematical Modelling and Numerical Analysis},
	year = 2023,
	volume = 57,
	number = 2,
	pages = {645-670}
}

@article{csirik2023coupled,
  title={Coupled-Cluster theory revisited-Part {II:} {A}nalysis of the single-reference Coupled-Cluster equations},
  author={Mih{\'a}ly A. Csirik and Andre Laestadius},
  journal={ESAIM: Mathematical Modelling and Numerical Analysis},
  volume={57},
  number={2},
  pages={545--583},
  year={2023},
  publisher={EDP Sciences}
}

@article{laestadius2019coupled,
  title={The coupled-cluster formalism--a mathematical perspective},
  author={Laestadius, Andre and Faulstich, Fabian M.},
  journal={Molecular Physics},
  volume={117},
  number={17},
  pages={2362--2373},
  year={2019},
  publisher={Taylor \& Francis}
}

@inproceedings{breiding2018homotopycontinuation,
  title={HomotopyContinuation. jl: A package for homotopy continuation in {J}ulia},
  author={Breiding, Paul and Timme, Sascha},
  booktitle={Mathematical Software--ICMS 2018: 6th International Conference, South Bend, IN, USA, July 24-27, 2018, Proceedings 6},
  pages={458--465},
  year={2018},
  organization={Springer}
}

@article{faulstich2024coupled,
  title={Coupled Cluster Theory: Toward an Algebraic Geometry Formulation},
  author={Faulstich, Fabian M. and Oster, Mathias},
  journal={SIAM Journal on Applied Algebra and Geometry},
  volume={8},
  number={1},
  pages={138--188},
  year={2024},
  publisher={SIAM}
}

@incollection{piecuch2000search,
  author    = "Piecuch, Piotr and Kowalski, Karol",
  title     = "In search of the relationship between multiple solutions characterizing coupled-cluster theories",
  booktitle = "Computational Chemistry, Reviews of Current Trends",
  publisher = "",
  pages={1--104},
  year      = "2000",
  editor    = "J. Leszczynski"
}

@article{kowalski2000complete2,
  title={Complete set of solutions of multireference coupled-cluster equations: The state-universal formalism},
  author={Kowalski, Karol and Piecuch, Piotr},
  journal={Physical Review A},
  volume={61},
  number={5},
  pages={052506},
  year={2000},
  publisher={APS}
}

@article{kowalski1998towards,
  title={Towards complete solutions to systems of nonlinear equations of many-electron theories},
  author={Kowalski, Karol and Jankowski, Karol},
  journal={Physical Review Letters},
  volume={81},
  number={6},
  pages={1195},
  year={1998},
  publisher={APS}
}

@article{kowalski2000complete,
  title={Complete set of solutions of the generalized {B}loch equation},
  author={Kowalski, Karol and Piecuch, Piotr},
  journal={International Journal of Quantum Chemistry},
  volume={80},
  number={4-5},
  pages={757--781},
  year={2000},
  publisher={Wiley Online Library}
}

@Article{vcivzek1966correlation,
	author={{\v{C}}{\'\i}{\v{z}}ek, Ji{\v{r}}{\'\i}},
	title={On the Correlation Problem in Atomic and Molecular Systems. calculation of wavefunction components in Ursell-type expansion using quantum-field theoretical methods},
	journal={The Journal of Chemical Physics},
	volume={45},
	number={11},
	pages={4256--4266},
	year={1966},
	publisher={AIP}
}

@Article{schneider2009analysis,
	author={Schneider, Reinhold},
	title={Analysis of the Projected Coupled Cluster Method in Electronic Structure Calculation},
	journal={Numerische Mathematik},
	volume={113},
	number={3},
	pages={433--471},
	year={2009},
	publisher={Springer}
}

@Article{rohwedder2013continuous,
	author={Rohwedder, Thorsten},
	title={The Continuous Coupled Cluster Formulation for the Electronic {S}chr{\"o}dinger Equation},
	journal={ESAIM: Mathematical Modelling and Numerical Analysis},
	volume={47},
	number={2},
	pages={421--447},
	year={2013},
	publisher={EDP Sciences}
}

@Article{rohwedder2013error,
	author={Rohwedder, Thorsten and Schneider, Reinhold},
	title={Error Estimates for the Coupled Cluster Method},
	journal={ESAIM: Mathematical Modelling and Numerical Analysis},
	volume={47},
	number={6},
	pages={1553--1582},
	year={2013},
	publisher={EDP Sciences}
}

@Book{helgaker2014molecular,
	title={Molecular Electronic-Structure Theory},
	author={Helgaker, Trygve and J\"orgensen, Poul and Olsen, Jeppe},
	year={2014},
	doi = {},
	publisher={John Wiley \& Sons}
}

@article{faulstich2019analysis,
  title={Analysis of the tailored coupled-cluster method in quantum chemistry},
  author={Faulstich, Fabian M. and Laestadius, Andre and Legeza, {\"O}rs and Schneider, Reinhold and Kvaal, Simen},
  journal={SIAM Journal on Numerical Analysis},
  volume={57},
  number={6},
  pages={2579--2607},
  year={2019},
  publisher={SIAM}
}

@article{vzivkovic1978analytic,
	title={Analytic connection between configuration-interaction and coupled-cluster solutions},
	author  = {Tomislav P. {\v{Z}}ivkovi{\'c} and Hendrik J. Monkhorst},
	journal={Journal of Mathematical Physics},
	volume={19},
	number={5},
	pages={1007--1022},
	year={1978},
	publisher={AIP}
}

@misc{breiding2025homotopyiterators,
      title={Homotopy Iterators}, 
      author={Paul Breiding and Taylor Brysiewicz and Hannah Friedman},
      year={2025},
      eprint={2509.08084},
      archivePrefix={arXiv},
      primaryClass={math.AG},
      url={https://arxiv.org/abs/2509.08084}, 
}

@article{faulstich2024recent,
  title={Recent mathematical advances in coupled cluster theory},
  author={Faulstich, Fabian M.},
  journal={International Journal of Quantum Chemistry},
  volume={124},
  number={13},
  pages={e27437},
  year={2024},
  publisher={Wiley Online Library}
}

@article{gesmundo2025collineation,
    author={Gesmundo, Fulvio and Keneshlou, Hanieh},
    title={Collineation varieties of tensors},
    journal={Collectanea Mathematica},
    year={2025},
    pages={1-21}
}

@book{fulton2013representation,
  title={Representation theory: a first course},
  author={Fulton, William and Harris, Joe},
  volume={129},
  year={2013},
  publisher={Springer Science \& Business Media}
}

@article{pauli1925zusammenhang,
  title={{\"U}ber den {Z}usammenhang des {A}bschlusses der {E}lektronengruppen im {A}tom mit der {K}omplexstruktur der {S}pektren},
  author={Pauli, Wolfgang},
  journal={Zeitschrift f{\"u}r Physik},
  volume={31},
  number={1},
  pages={765--783},
  year={1925},
  publisher={Springer}
}

@article{schrodinger1926undulatory,
  title={An undulatory theory of the mechanics of atoms and molecules},
  author={Schr{\"o}dinger, Erwin},
  journal={Physical review},
  volume={28},
  number={6},
  pages={1049},
  year={1926},
  publisher={APS}
}

@misc{LiE,
author={van Leeuwen, Marc A. A. and Cohen, Arjeh M. and Lisser, Bert},
title ={{LiE}, A Package for {L}ie Group Computations},
journal={Computer Algebra Nederland, Amsterdam}, 
ISBN={90-74116-02-7, 1992 },
howpublished="\url{http://www-math.univ-poitiers.fr/~maavl/LiE/}",
}

@article{sverrisdottir2025algebraic,
  title={Algebraic Varieties in Second Quantization},
  author={Sverrisd{\'o}ttir, Svala},
  journal={arXiv:2505.17276},
  year={2025}
}

@article{kalman2025,
  title={Nonlinear {K}alman varieties},
  author={Flavio Salizzoni and Luca Sodomaco and Julian Weigert},
  journal={arXiv:2512.16540v1},
  year={2025}
}

@misc{Faulstich2025SupplementaryCode,
  author       = {Faulstich, Fabian M. and Galgano, Vincenzo and Neuhaus, Elke and Portakal, Irem},
  title        = {Supplementary code for ``{T}he {C}oupled {C}luster {D}oubles {T}runcation {V}ariety of {F}our {E}lectrons"},
  year         = {2026},
  journal    = {Zenodo},
  doi          = {10.5281/zenodo.17912500},
  howpublished="\url{https://doi.org/10.5281/zenodo.17912500}"
}

@article{faulstich2024algebraic,
  title={Algebraic varieties in quantum chemistry},
  author={Faulstich, Fabian M and Sturmfels, Bernd and Sverrisd{\'o}ttir, Svala},
  journal={Foundations of Computational Mathematics},
  volume={25},
  number={4},
  pages={1167--1198},
  year={2025},
  publisher={Springer}
}

@article{borovik2024coupled,
title = {Coupled cluster degree of the {G}rassmannian},
journal = {Journal of Symbolic Computation},
volume = {128},
pages = {102396},
year = {2025},
issn = {0747-7171},
author = {Viktoriia Borovik and Bernd Sturmfels and Svala Sverrisdóttir}
}

@book{landsberg2011tensors,
  title={Tensors: geometry and applications: geometry and applications},
  author={Landsberg, Joseph M},
  volume={128},
  year={2011},
  publisher={American Mathematical Soc.}
}

@article{brysiewicz2023lawrence,
  title={{L}awrence lifts, matroids, and maximum likelihood degrees},
  author={Brysiewicz, Taylor and Maraj, Aida},
  journal={Algebraic Statistics},
  volume={16},
  number={2},
  pages={217--242},
  year={2025},
  publisher={Mathematical Sciences Publishers}
}

@article{brysiewicz2019degree,
author = {Brysiewicz, Taylor and Gesmundo, Fulvio},
title = {The degree of {S}tiefel manifolds},
journal = {Enumerative Combinatorics and Applications},
volume = {1},
number = {3},
year = {2021}
}

@book{shafarevich2013basic,
  title={Basic Algebraic Geometry 1: Varieties in Projective Space},
  author={Shafarevich, Igot R. and Reid, Miles},
  isbn={9783642379567},
  series={SpringerLink~:~B{\"u}cher},
  year={2013},
  publisher={Springer Berlin Heidelberg}
}

@book{harris1992algebraic,
  title={Algebraic Geometry: A First Course},
  author={Harris, Joe},
  isbn={9780387977164},
  lccn={96105844},
  series={Graduate Texts in Mathematics},
  year={1992},
  publisher={Springer}
}

\end{document}